\documentclass[reqno]{amsart}
\usepackage{bbm, comment}

\usepackage[T1]{fontenc}
\usepackage{dsfont}
\usepackage{mathrsfs}
\usepackage[colorlinks]{hyperref}
\usepackage{xcolor}
\usepackage[a4paper,asymmetric]{geometry}
\usepackage{mathscinet}
\usepackage{latexsym}
\usepackage{amsthm}
\usepackage{amssymb}
\usepackage{amsfonts}
\usepackage{amsmath}
\usepackage{longtable}
\usepackage{graphicx}
\usepackage{multirow}
\usepackage{multicol}

\newtheorem{theorem}{Theorem}[section]
\newtheorem{thm}[theorem]{Theorem}
\newtheorem{lemma}[theorem]{Lemma}
\newtheorem{lem}[theorem]{Lemma}
\newtheorem{remark}[theorem]{Remark}
\newtheorem{proposition}[theorem]{Proposition}

\newtheorem{corollary}[theorem]{Corollary}

\newtheorem{hyp}[theorem]{HYPOTHESIS}

\theoremstyle{definition}

\newtheorem{defn}[theorem]{Definition}
\newtheorem{example}[theorem]{Example}
 
\theoremstyle{remark}
\numberwithin{equation}{section}

 \DeclareMathAlphabet{\mathpzc}{OT1}{pzc}{m}{it}
 \DeclareMathAlphabet{\mathsfsl}{OT1}{cmss}{m}{sl}

  \newcommand{\FH}{\mathfrak{H}}

\newcommand{\dif}{\mathrm{d}}

\newcommand{\abs}[1]{\left\vert#1\right\vert}
\newcommand{\set}[1]{\left\{#1\right\}}

\newcommand{\norm}[1]{\left\Vert#1\right\Vert}

\newcommand{\E}{\mathbb{E}}

\newcommand{\R}{\mathbb{R}}

 \newcommand{\innp}[1]{\langle {#1}\rangle}
\newcommand{\Be}{\begin{equation}}
\newcommand{\Ee}{\end{equation}}
\newcommand{\Bs}{\begin{split}}
\newcommand{\Es}{\end{split}}
\newcommand{\Bes}{\begin{equation*}}
\newcommand{\Ees}{\end{equation*}}
\newcommand{\BT}{\begin{thm}}
\newcommand{\ET}{\end{thm}}
\newcommand{\Bp}{\begin{proof}}
\newcommand{\Ep}{\end{proof}}
\newcommand{\BL}{\begin{lem}}
\newcommand{\EL}{\end{lem}}
\newcommand{\BP}{\begin{proposition}}
\newcommand{\EP}{\end{proposition}}
\newcommand{\BC}{\begin{corollary}}
\newcommand{\EC}{\end{corollary}}
\newcommand{\BR}{\begin{remark}}
\newcommand{\ER}{\end{remark}}
\newcommand{\BD}{\begin{defn}}
\newcommand{\ED}{\end{defn}}
\newcommand{\BI}{\begin{itemize}}
\newcommand{\EI}{\end{itemize}}

\newcommand\sgn{\mathrm{sgn}}

\allowdisplaybreaks

\begin{document}
\title[Statistical Estimations for Non-Ergodic Vasicek Model]{Statistical Estimations for Non-Ergodic Vasicek Model Driven by Two Types of Gaussian Processes}

\author[Y. Chen]{Yong Chen}
\address{School of Mathematics and Statistics, Jiangxi Normal University, 330022, Nanchang,  P. R. China}
\author{Wujun  Gao}
\address{College of Big Data and Internet, Shenzhen Technology University, Shenzhen {518118}, \ China.}
\author{{Ying   {LI}} }
\address{ School of Mathematics and Computational Science, Xiangtan University, Xiangtan {411105},\ China  }\email{ liying@xtu.edu.cn}
\begin{abstract}
 We study the joint asymptotic distribution of the least squares estimator of the parameter $(\theta,\,\mu)$ for the non-ergodic Vasicek models driven by seven specific Gaussian processes. 
To facilitate the proofs, we derive the inner product formulas of the canonical Hilbert spaces associated to the seven specific Gaussian processes. The integration by parts for normalized bounded variation functions is essential to the inner product formulas. We apply the inner product formulas of the seven Gaussian processes to check the set of conditions of Es-Sebaiy, Es.Sebaiy (2021).
\end{abstract}
\maketitle

\section{Introduction and main results}\label{sec introd}
The Gaussian Vasicek-type model is known as the solution to the stochastic
differential equation
\begin{align}\label{Vasicekmodel0000}
  \dif X_t =\theta(\mu +X_t) \dif t +\dif  G_t, \quad X_0 = 0 ,\quad \theta,\mu\in \R,
\end{align}
where the {driving} noise $G=\set{G_t: t\in[0,T ]}$ is a Gaussian process. This model has a wide range of applications in many fields,
such as economics, finance, biology, medical and environmental sciences. In the economic field, it has been used to describe the fluctuation of interest rates, please refer to \cite{huang 2012}. In the financial field, it can also be used as a random investment model, {see \cite{wu2020}}.

When $ \theta>0$ (i.e., non-ergodic case), the statistical estimations for the Gaussian Vasicek-type model \eqref{Vasicekmodel0000}  
have been studied in \cite{Yuqian2020} and \cite{Sebaiy22021} {based on continuous observations of $X$ over the time interval $[0, T]$}. Under a set of conditions, the asymptotic behavior of the estimators of the parameters $\theta$ and $\mu$ are obtained in \cite{Yuqian2020} and the joint asymptotic distributions of the estimators of the parameters $(\theta,\,\mu)$ and $(\theta,\,\alpha)$ with $\alpha=\theta\mu$ are derived in \cite{Sebaiy22021}  (see Theorem~\ref{Eseiby theorem}). The fractional Brownian motion (fBm), the bi-fractional Brownian motion and sub-fractional Brownian motion are examples satisfying the set of conditions. Clearly, the result of \cite{Sebaiy22021} is stronger than that of \cite{Yuqian2020}. 
When $G$ is the generalized sub-fractional Brownian motion with parameters $H',K\in (0,1)$, the asymptotic behavior of the estimators of $\theta$ and $\mu$ are shown in \cite{KX2023}. 
Non-Gaussian Vasicek-type models driven by Hermite process and Liu process are investigated in {\cite{NT 19,Yuqian2020}} and \cite{wei2023} respectively. {Moreover, based on discrete observations of $X$ over the time interval $[0, T]$,  the asymptotic properties of the least squares estimators of $(\theta, \mu)$ are investigated for stationary case in \cite{xiaoyu2019a} and 
for explosive case in \cite{jiang 2014} when $G$ is a fractional Brownian motion. }

The aim of this paper is to study the statistical estimation for a non-ergodic Gaussian Vasicek-type model
\begin{align}\label{Vasicekmodel}
  \dif X_t =\theta(\mu +X_t) \dif t +\dif  G_t, \quad X_0 = 0 ,\quad \theta>0,\, \mu\in \R,
\end{align}
where the {driving} Gaussian noise $G=(G_t)_{t\in[0,T ]}$ is taken from Examples~\ref{exmp 000001}-\ref{exmp lizi001}. 
There are two reasons for us to propose this problem. One is that the study of the Vasicek models driven by Examples~\ref{exmp 000001}-\ref{exmp lizi001} is novel, the other is that the set of conditions $(\mathcal{A}_3)$-$(\mathcal{A}_5)$ of Theorem~\ref{Eseiby theorem} are not given in terms of the covariance structure of $G$. 


\begin{example}\label{exmp 000001}
{ The Gaussian process $G=\set{G_t:t\ge 0}$ is defined by the Wiener integral: 
  $$G_{t}  = \sqrt{C_H}\int_{0}^{\infty} \left( 1-e^{-r t} \right) r^{-\frac {1+2H}{2}} \dif W_{r} ,\quad t\ge 0,$$ where $\{W_t,t \ge 0 \}$ is the standard Brownian motion, $H \in (0, \frac12)\cup(\frac12,1) $, and the constant $C_H$ is given by
  \begin{align*}
    C_H=\left\{
    \begin{array}{ll}
      \frac{H}{\Gamma(1-2H)},        & \quad  H\in (0,\frac12), \\
      \frac{H(2H-1)} {\Gamma(2-2H)}, & \quad H\in (\frac12,1),
    \end{array}
    \right.
  \end{align*}
  Then the covariance function of $G$ is given by
  \begin{equation*}
    R(s,t)=\left\{
    \begin{array}{ll}
      \frac12[t^{2H}+s^{2H} -(t+s)^{2H}], & \quad  H\in (0,\frac12), \\
      \frac12[(t+s)^{2H}-t^{2H}-s^{2H} ], & \quad H\in (\frac12,1).
    \end{array}
    \right.
  \end{equation*}
  See \cite{Bardina2009, Lei2009}.}
\end{example}

\begin{example}\label{exmp 000001-04}
 { The trifractional Brownian motion $Z^{H',K}=\set{Z^{H',K}(t): t\ge 0} $ with parameters $H' \in (0, 1),\,K \in(0,1)$ is a centered Gaussian process with covariance function
  $$ R(s,\, t)=t^{2H'K} + s^{2H'K}- (t^{2H'}+s^{2H'})^{K}  .$$
  See \cite{Lei2009, ma2013}. }
\end{example}


\begin{example} \label{exmp 000001-06}
{  The process $G=\{G_t , t\ge 0\} $ with parameter $H\in (0,\frac12]$ is a centered Gaussian process with covariance function
  $$ R(s, t)= \frac12\left[(t+s)^{2H}-(\max(s,t))^{2H} \right] . $$
 This function is nonnegative definite if and only if $H\in (0,\frac12]$. See \cite[Theorem 1.1(i)]{Talarczyk2020}.}
\end{example}

\begin{example}\label{exmp 000001-06-buch}
{  The process $\{Z_t , t\ge 0\} $ with parameter $H\in (0,\frac12]$ is a centered Gaussian process with covariance function
  $$ R(s, t)= \Gamma(1-2H)(\min(s,t))^{2H} . $$
  {This function is nonnegative definite if and only if $H\in (0,\frac12]$.}
  See \cite[Theorem 2.1]{DW2016}.}
\end{example}

\begin{example}\label{exmp0005}
{  The generalized sub-fractional Brownian motion (also known as the sub-bifractional Brownian motion) $S^{H',K}=\set{S^{H',K}(t), t\ge 0}$ with parameters $H' \in (0, 1),\,K \in(0,2)$ such that $H:=H'K\in (0,1)$ 
is a centered Gaussian process with covariance function}
  $$ R(s,\, t)= (s^{2H'}+t^{2H'})^{K}-\frac12 \left[(t+s)^{2H'K} + \abs{t-s}^{2H'K} \right].$$
 When $K=1$, it degenerates to the sub-fractional Brownian motion $S^H(t) $ with parameter $H\in (0,1)$.  See \cite{NoutyJourn2013} and \cite{Sgh2013} for the case of $K\in (0,1)$ and $K\in (1,2)$, respectively.
\end{example}
\begin{example}\label{exmp7-1zili}
  The generalized fractional Brownian motion { $G =\{G_t , t\ge 0\}$ is a centered Gaussian process with covariance function}
  $$R(s,\, t)=\frac{(a+b)^2}{2(a^2+b^2)}(s^{2H}+t^{2H})-\frac{ab}{a^2+b^2}(s+t)^{2H}-\frac12 \abs{t-s}^{2H},$$
  where $H\in (0, 1)$ and $(a,b)\neq (0,0)$. See \cite{Zili17}. It is an extension of both the fractional Brownian motion and the sub-fractional Brownian motion. 
\end{example}
\begin{example}\label{exmp lizi001}
  The Gaussian process $G =\{G_t , t\ge 0\}$ {is a centered Gaussian process with } covariance function
  \begin{align} \label{cova func 01}
    R(s, t)=\frac12\big[ (\max(s,t))^{2H} - |t-s|^{2H}\big] .
  \end{align}
{This function is nonnegative definite if and only if $H\in (0,\frac12]$.}  See \cite[Theorem 1.1 (ii)]{Talarczyk2020}.
\end{example}

Now let us return to the non-ergodic Gaussian Vasicek-type model \eqref{Vasicekmodel}. Assume the whole trajectory of $X=\set{X_t: t\in [0, T]}$ is continuously observed. The least squares estimators are proposed in \cite{Sebaiy22021} for the unknown parameters $\theta$, $\alpha=\theta\mu$ and $\mu $ as follows:
\begin{align}\label{LSE1}
  \hat{\theta}_T =\frac{\frac 12 T X_T^2-X_T \int_{0}^{T}X_t \dif t }{T\int_{0}^{T}X_t^2 \dif t-\left(\int_{0}^{T}X_t \dif t \right)^2},
\end{align}
\begin{align}\label{LSE2}
  \hat{\alpha}_T=\frac{X_T\int_{0}^{T}X_t^2 \dif t - \frac 12 X_T^2 \int_{0}^{T}X_t \dif t }{T\int_{0}^{T}X_t^2 \dif t-\left( \int_{0}^{T}X_t \dif t \right)^2},
\end{align}
and \begin{align}\label{LSE3}
  \hat{\mu}_T=\frac{\hat{\alpha}_T}  {\hat{\theta}_T } =\frac{\int_{0}^{T}X_t^2 \dif t - \frac 12 X_T \int_{0}^{T}X_t \dif t }{\frac 12 T X_T- \int_{0}^{T}X_t \dif t }.
\end{align}

In this paper, we will demonstrate the joint asymptotic distributions of the estimators $(\hat{\theta}_T,\,\hat{\mu}_T)$ and $(\hat{\theta}_T,\,\hat{\alpha}_T)$ when the {driving} noise $G$ is taken from Examples~\ref{exmp 000001}-\ref{exmp lizi001}. The results 
are stated in the following Theorem~\ref{asymptoticth}. By \eqref{beta biaoshishizi},  three types of fractional powers emerge in the asymptotic behavior of $\hat{\theta}_T$. 

\begin{theorem}\label{asymptoticth}
For the Vasicek model \eqref{Vasicekmodel} driven by the Gaussian process given by Examples~\ref{exmp 000001}-\ref{exmp lizi001}, we have that as $T \to \infty$,
  \begin{align*}
     & \left(T^{\beta} e ^{\theta T}(\hat{\theta}_T - \theta), \, T^{1-H}(\hat{\mu}_T - \mu) \right) \xrightarrow{law} \left( \frac{ N_2}{N_3},\, \frac {1}{\theta} N_1 \right), \\
     & \left( T^{\beta} e ^{\theta T}(\hat{\theta}_T - \theta), \, T^{1-H}(\hat{\alpha}_T - \alpha)\right) \xrightarrow{law} \left(\frac{ N_2}{N_3},\, N_1 \right),
  \end{align*} where
  \begin{equation}\label{beta biaoshishizi}
    \beta=\left\{
    \begin{array}{ll}
      1-H,           & \quad \text{for Examples~\ref{exmp 000001}-\ref{exmp 000001-04}}  ,       \\
      \frac{1}{2}-H, & \quad \text{for Examples~\ref{exmp 000001-06}-\ref{exmp 000001-06-buch}}, \\
      0,             & \quad\text{for Examples~\ref{exmp0005}-\ref{exmp lizi001}},
    \end{array}
    \right.
  \end{equation}
  and $N_1\sim \mathcal{N}(0,\lambda_G^2), N_2\sim\mathcal{N}(0, 4\theta^2\sigma_G^2), N_3\sim  \mathcal{N}(\mu,\theta^2\sigma^2_{\infty})$  are independent Gaussian random variables and $ \lambda_G^2,\,\sigma_{G}^2,\,\sigma^2_{\infty}$ are positive constants given by \eqref{assum51-new}, \eqref{changshu sg2} and \eqref{Sigmapingfang 2} respectively.
\end{theorem}
{
\begin{remark}\label{asymptoticth OU}
The non-ergodic Ornstein–Uhlenbeck (OU) process is the solution to the stochastic differential
equation
\begin{align}\label{OU-model}
  \dif Y_t =\theta Y_t \dif t +\dif  G_t, \quad Y_0 = 0 ,\quad \theta>0.
\end{align}
The least squares estimator for the drift parameter $\theta$ is proposed in \cite{Sebaiy2016} as follows:
\begin{align}\label{OU estimator dingyi}
  \tilde{\theta}_T=\frac{Y_T^2}{2\int_0^T Y_t^2 \dif t}.
\end{align} The asymptotic behavior of $\tilde{\theta}_T$ is a special case of Theorem~\ref{asymptoticth} essentially.
  If the {driving} process $G=\set{G_t: t\in[0,T ]}$ of the OU process \eqref{OU-model} is taken from Examples~\ref{exmp 000001}-\ref{exmp lizi001}, we have that as $T \to \infty$,
  \begin{align}\label{ou jiexianfubu}
    T^{\beta} e ^{\theta T}(\tilde{\theta}_T - \theta) \xrightarrow{law} \frac{2\sigma_G}{\sigma_{\infty}}\mathcal{C}(1),
  \end{align}
  where $\mathcal{C}(1)$ is the standard Cauchy distribution, and $\beta,\,\sigma_G^2,\,\sigma_{\infty}^2$ are positive constants given by \eqref{beta biaoshishizi}, \eqref{changshu sg2} and \eqref{Sigmapingfang 2} respectively. Also, three types of fractional powers emerge in the asymptotic behavior of $\tilde{\theta}_T$.
\end{remark}}


{To check the validity of the assumptions $(\mathcal{A}_1)$-$(\mathcal{A}_5)$ of \cite{Sebaiy22021}
for the seven Gaussian driving processes given in Examples~\ref{exmp 000001}-\ref{exmp lizi001}, we have to study the canonical Hilbert spaces associated to Examples~\ref{exmp 000001}-\ref{exmp lizi001}. The second contribution of this paper is to derive the inner product formulas of the canonical Hilbert spaces associated to Examples~\ref{exmp 000001}-\ref{exmp lizi001}. The result is summarized in Proposition~\ref{prop 2-4 fenxilizi}.}

\begin{proposition}\label{prop 2-4 fenxilizi}
Let $f,\, g\in \mathcal{V}_{[0,T]}$. {Denote by $\mathfrak{H}$ and $\mathfrak{H}_1$ the canonical Hilbert space associated to $G$ and to the fBm $B^H$ respectively.}
\begin{description}
   \item [(i)] For Examples~\ref{exmp 000001}-\ref{exmp 000001-04}, the inner product formula    \begin{align}\label{diyineijigongshi 00}
        \langle f,\,g \rangle_{\FH}=\int_0^T  f(t)\dif t \int_0^T g(s)  \frac{\partial^2 R(s,t)}{\partial s  \partial t } \dif s
  \end{align}
   holds, where  
\begin{equation}\label{ziziexmp11111}
    \frac{\partial^2 R(s,t)}{\partial s  \partial t } =H\abs{2H-1}  (t+s)^{2H-2}
\end{equation}
and \begin{equation}\label{liziexmp2222}
    \frac{\partial^2 R(s,t)}{\partial s  \partial t } = (2H')^2 K (1-K) (ts)^{2H'-1}  (t^{2H'}+s^{2H'})^{K-2} 
\end{equation}
for Examples~\ref{exmp 000001}-\ref{exmp 000001-04} respectively.
   \item [(ii)] For Example~\ref{exmp 000001-06}, the inner product formula \eqref{innp fg3-zhicheng0-0} is 
   \begin{align}\label{guanjiandengshi 000-new1-000}
  \langle f,\,g \rangle_{\FH} & =H(2H-1)\int_{[0,T]^2}f(t)g(s) (t+s)^{2H-2} \dif t \dif s +H \int_0^T  {f}(t) g(t)t^{2H-1}\dif t .\end{align} 
   \item [(iii)] For Example~\ref{exmp 000001-06-buch}, the inner product formula \eqref{innp fg3-zhicheng0-0} is  
   \begin{align}\label{guanjiandengshi 000-new1-000-buch}
  \langle f,\,g \rangle_{\FH} & =2H\Gamma(1-2H) \int_0^T  {f}(t) g(t)t^{2H-1}\dif t.\end{align} 
   \item [(iv)] For Examples~\ref{exmp0005}-\ref{exmp7-1zili}, the inner product formula 
    \begin{align}\label{dierneijigongshi 00}
       \langle f,\,g \rangle_{\FH}-\langle f,\,g \rangle_{\FH_1}=\int_0^T  f(t)\dif t \int_0^T g(s)  \frac{\partial }{\partial s}\left( \frac{\partial R(s,t)}{\partial t} -  \frac{\partial R^B(s,t)}{\partial t} \right)  \dif s.
\end{align} 
   holds, where 
\begin{equation}\label{emp 1.5- pdshu}
\begin{split}
    &\frac{\partial }{\partial s}\left( \frac{\partial R(s,t)}{\partial t} -  \frac{\partial R^B(s,t)}{\partial t} \right)\\
    &=H'K  \big[4H'(K-1)(s^{2H'}+t^{2H'})^{K-2}(ts)^{2H'-1} -(2H'K-1)(s+t)^{2H'K-2}\big]
    \end{split}
\end{equation}
   and 
   \begin{equation}\label{emp 1.6- pdshu}
    \frac{\partial }{\partial s}\left( \frac{\partial R(s,t)}{\partial t} -  \frac{\partial R^B(s,t)}{\partial t} \right) =-\frac{2H(2H-1)ab}{a^2+b^2}(t+s)^{2H-2}
\end{equation}
for Examples~\ref{exmp0005}-\ref{exmp7-1zili} respectively.
    \item [(v)] For Example~\ref{exmp lizi001}, the inner product formula \eqref{innp fg3-zhicheng0} is   \begin{align}\label{guanjiandengshi 000-new1}
  \langle f,\,g \rangle_{\FH}-  \langle f,\,g \rangle_{\FH_1} & =-H \int_0^T  {f}(t) g(t)t^{2H-1}\dif t ,\end{align}
which implies that if the intersection of the supports of $f,\,g$ has Lebesgue measure zero, then
\begin{align} \label{innp fg3-zhicheng0-00}
  \langle f,\,g \rangle_{\FH}=\langle f,\,g \rangle_{\FH_1}=H(2H-1)\int_{[0,T]^2}  f(t)g(s) \abs{t-s}^{2H-2} \dif t   \dif s. 
\end{align}
   \end{description}
\end{proposition}
\begin{remark}
    We can obtain a formula similar to (iv) for the sub-fractional Brownian motion and the bi-fractional Brownian motion.
\end{remark}
The rest of the paper is organized as follows: {In  Section~\ref{sec-002} we introduce the canonical Hilbert space associated to $G$ and present the theorem about the joint asymptotic distributions of the estimators of the parameters $(\theta,\,\mu)$ and $(\theta,\,\alpha)$ with $\alpha=\theta\mu$ in \cite{Sebaiy22021}.
Section~\ref{new-section} }is dedicated to show Proposition~\ref{prop 2-4 fenxilizi}. In Section~\ref{proof mn results}, we prove Theorems~\ref{asymptoticth} based on Proposition~\ref{prop 2-4 fenxilizi} and the theorem extracted from \cite{Sebaiy22021} (see Section~\ref{sub.sec.2.1}).

The symbol $C$ represents a generic constant, the value of which may vary from one line to another.

\section{Preliminary}\label{sec-002}

Assume that $G$ is defined on a complete probability space $(\Omega, \mathcal{F}, P)$. Let $\mathfrak{H}$ denote the {canonical Hilbert space associated to $G$}, which is defined as the closure of the space of all real-valued step functions on $[0, T]$ with the inner product
\begin{align*}
  \langle \mathbbm{1}_{[a,b]},\,\mathbbm{1}_{[c,d]}\rangle_{\FH}=\E\left(( G_b-G_a) ( G_d-G_c) \right).
\end{align*}
By abuse of notation, we use the same symbol $$G=\left\{G(h)=\int_{[0,T]}h(t)\dif G_t, \quad h \in \mathfrak{H}\right\}$$ to represent the isonormal Gaussian process on the probability space $(\Omega, \mathcal{F}, P)$. This process is indexed by the elements in the Hilbert space $\mathfrak{H}$ and satisfies It\^{o}'s isometry:
\begin{align}\label{G extension defn}
  \mathbb{E}(G(g)G(h)) = \langle g, h \rangle_{\mathfrak{H}}, \quad
  \forall g, h \in \mathfrak{H}.
\end{align}
Theorem 2.3 of \cite{Jolis 2007} gives the following inner product formula for $\mathfrak{H}$:
 \begin{align}\label{jolis 00}
 \langle f,\,g \rangle_{\FH}=\int_{[0,T]^2}  R(s,t) \dif \big(\nu_{f}\times \nu_{g}\big)(s, t),\qquad \forall f,\, g\in \mathcal{V}_{[0,T]},
 \end{align}where $\nu_{f}\times \nu_{g}$ is the product measure of $\nu_f$ and $\nu_{g}$,  
and for each $g\in \mathcal{V}_{[0,T]}$, $\nu_{g}$ is the restriction to $([0,T ], \mathcal{B}([0,T ]))$ of the Lebesgue-Stieltjes signed measure $\mu_{g^0}$ on $\left(\R,\mathcal{B}(\R)\right)$ of $g^0$ which is defined as
\begin{equation}\label{g0dingyi}
  g^0(x)=\left\{
  \begin{array}{ll}
    g(x), & \quad \text{if } x\in [0,T] , \\
    0,    & \quad \text{otherwise}.
  \end{array}
  \right.
\end{equation} 
The measure $\nu_g$ is similarly defined as in \cite{Jolis 2007}.

\subsection{the asymptotic behavior of the three parameters $\theta,\,\mu$ and $\alpha=\mu \theta$ }\label{sub.sec.2.1}
The following theorems are extracted from \cite{Sebaiy22021} and \cite{Sebaiy2016} respectively. 
\begin{theorem} \label{Eseiby theorem}
Let the Gaussian Vasicek-type mode be given by \eqref{Vasicekmodel} and the estimators  $\hat{\theta}_T$,\, $\hat{\mu}_T$ and $\hat{\alpha}_T$ be given by \eqref{LSE1}-\eqref{LSE3} and the normal distribution random variables $N_i,\,i=1,2,3$ be given by Theorem~\ref{asymptoticth}. Suppose that 
\begin{description}
   \item [($\mathcal{A}_1$)] There exist constants $c>0$ and  $\gamma\in (0,1)$ such that for every $s,t \ge 0$, $$\E[(G_t-G_s)^2]\le c\abs{t-s}^{2\gamma}.$$
   \item [($\mathcal{A}_2$)] There exist $\lambda_G>0$ and $p\in (0,1)$ such that $\lim\limits_{T \to \infty} \frac{\E[G_T^2]}{T^{2p}}=\lambda_G^2.$
   \item [($\mathcal{A}_3$)] There exist constants $\beta\ge 0$ and $\sigma_G>0$ such that 
   $$\lim\limits_{T \to \infty} T^{2\beta}\E\Big[\big( \int_{0}^T e^{-\theta (T-s)}\dif G_s\big)^2\Big]=\sigma_G^2.$$
   \item [($\mathcal{A}_4$)]For all fixed $s\ge 0$,
   $$ \lim _{T \to \infty} \E\Big[G_s \int_{0}^T e^{-\theta (T-r)}\dif G_r\Big]=0.$$
   \item [($\mathcal{A}_5$)]For all fixed $s\ge 0$,
    $$ \lim _{T \to \infty}\frac{\E[G_sG_T]}{T^{p}}=0,\qquad \lim _{T \to \infty} \frac{1}{T^{p}}\E\Big[G_T \int_{0}^T e^{-\theta (T-r)}\dif G_r\Big]=0,$$
   {where $p$ is the positive constant given in $(\mathcal{A}_2)$.}
\end{description}
Then we have 
\begin{description}
   \item [(i)]All the estimators $\hat{\theta}_T$,\, $\hat{\mu}_T$ and $\hat{\alpha}_T$ are strong consistent under the assumption $(\mathcal{A}_1)$.
   \item [(ii)]If $(\mathcal{A}_1)$-$(\mathcal{A}_4)$ hold then 
   as $T\to\infty$,
   $$T^{\beta} e ^{\theta T}(\hat{\theta}_T - \theta) \xrightarrow{law}\frac{N_2}{N_3},\quad  T^{1-p}(\hat{\mu}_T - \mu)   \xrightarrow{law} \frac{1}{\theta} N_1, \quad T^{1-p}(\hat{\alpha}_T - \alpha)   \xrightarrow{law}   N_1.$$
   \item [(iii)]Moreover, if $(\mathcal{A}_5)$ holds then 
   as $T\to\infty$,
     \begin{align*}
     & \left( T^{\beta} e ^{\theta T}(\hat{\theta}_T - \theta), \, T^{1-p}(\hat{\mu}_T - \mu) \right) \xrightarrow{law} \left( \frac{ N_2}{N_3},\, \frac {1}{\theta} N_1 \right), \\
     & \left( T^{\beta} e ^{\theta T}(\hat{\theta}_T - \theta), \, T^{1-p}(\hat{\alpha}_T - \alpha)\right) \xrightarrow{law} \left(\frac{ N_2}{N_3},\, N_1 \right).
  \end{align*} 
   \end{description}
\end{theorem}
\begin{remark}
Theorem~\ref{Eseiby theorem} is an extension of \cite{Sebaiy22021} where they assume $\beta=0$.  Since this extension is minor and can be shown with a slight adaption of the original proof in \cite{Sebaiy22021}, we ignore the proof. The readers can refer to \cite{Sebaiy22021} for details. 
\end{remark}

\begin{theorem}\label{Esseiby 2}
Let the non-ergodic OU process be given by \eqref{OU-model} and the estimator  $\tilde{\theta}_T$ by \eqref{OU estimator dingyi} and the constants $\sigma_G^2,\,\sigma_{\infty}^2$ by Theorem~\ref{asymptoticth OU}. Suppose that   
\begin{description}
   \item [($\mathcal{H}_1$)] The process $G$ has H\"older continuous paths of order $\gamma\in (0,1]$.
      \item [($\mathcal{H}_2$)] There exist positive constants $c$ and $ q $ such that $\E[G_t ^2]\le c t^{2q}$ for every $t \ge 0$.
   \end{description}
Then we have 
\begin{description}
   \item [(i)]The estimator $\hat{\theta}_T$ is strong consistent under the assumption $(\mathcal{H}_1)$-$(\mathcal{H}_2)$.
   \item [(ii)]Moreover, if $(\mathcal{A}_3)$-$(\mathcal{A}_4)$ hold then  as $T\to\infty$,
     \begin{align*}
T^{\beta} e ^{\theta T}(\tilde{\theta}_T - \theta) \xrightarrow{law} \frac{2\sigma_G}{\sigma_{\infty}}\mathcal{C}(1),
  \end{align*}
  where $\mathcal{C}(1)$ is the standard Cauchy distribution.
   \end{description}
\end{theorem}
\begin{remark}
Theorem~\ref{Esseiby 2} is a minor extension of  \cite{Sebaiy2016} where they assume $\beta=0$.
\end{remark}

\section{{Proof of Proposition~\ref{prop 2-4 fenxilizi} }}\label{new-section}
{For the reader's convenience, we divide Section 3 into three subsections. In subsection~\ref{fenbujfgsh subsect} , we list two integration by parts formulas related to the measure $\nu_{g}$ (see Lemma~\ref{partial integral} and Lemma~\ref{prop 2-1}) from which the formula \eqref{jolis 00} can be simplified greatly. Subsection~\ref{gjzhenm sec.3.2} is devoted to the proof of Proposition~\ref{prop 2-4 fenxilizi}. In subsection~\ref{sub.sec.3.3}, we show Lemma~\ref{partial integral} and Lemma~\ref{prop 2-1}.} 
\subsection{some lemmas about integration by parts formula related to $\nu_{g}$}\label{fenbujfgsh subsect}

\begin{lemma}\label{partial integral}
Suppose $f,\,g$ are of normalized bounded variation on $[0,T]$ and the measure $\nu_g$ is given as above. If there are no points in $(0,T]$ where both $f$ and $g$ are discontinuous, then we have
 \begin{align}\label{01001}
	-\int_{[0,T]} g(t) \dif \mu_{f}(t)=\int_{[0,T]} f(t) \dif \nu_{ g}( t), 
\end{align}
where $\mu_{f}$ is the restriction to $([0,T], \mathcal{B}([0,T ]))$ of the Lebesgue-Stieltjes signed measure $\mu_{\bar{f}}$ on $\left(\R,\mathcal{B}(\R)\right)$ associated with $\bar{f}$ which is defined as
\begin{equation}\label{f01dingyi}
 \bar{f}(x)=\left\{
  \begin{array}{ll}
  0,  & \quad \text{if } x<-1,\\
   f(0)(x+1),    & \quad \text{if } x\in [-1,0),\\
    f(x), & \quad \text{if } x\in [0,T] , \\
   f(T), & \quad \text{if } x> T.
  \end{array}
  \right.
\end{equation} 
Especially, if $f$ is absolutely continuous on~$[0,T]$ then \eqref{01001} degenerates to 
\begin{align}   
-\int_0^T g(t) f'(t)\dif t=\int_{[0,T]} f(t) \dif\nu_{ g}( t). \label{01} \end{align}
\end{lemma}
\begin{remark}
The two measures $\mu_f$ and $\nu_{ g}$ are not symmetric in \eqref{01001}, which is due to the different extension methods of $f$ and $g$, see \eqref{f01dingyi} and \eqref{g0dingyi} respectively. The measure $\mu_{g^0}$ and hence its restriction $\nu_g$ may have atoms, i.e.,
  \begin{align}\label{wuyuanzi}
 \nu_g(\{0\})=g(0),\quad
   \nu_g(\{T\})=-g(T-),
 \end{align} but as a comparison, we have
   \begin{align}\label{wuyuanzi}
 \mu_f(\{0\})=0,\quad
   \mu_f(\{T\})=f(T)-f(T-),
 \end{align}
see the below proof for details. 
\end{remark}

To simplify the proof of  Proposition~\ref{prop 2-4 fenxilizi}, we first extract two common hypotheses (see Hypotheses~\ref{hypthe1-1} and \ref{hypthe2new-new}) from the covariance functions $R(s,t)$ of Examples~\ref{exmp 000001}-\ref{exmp lizi001} and then develop two inner product formulas for functions of bounded variation in the {canonical Hilbert spaces associated to} the two general Gaussian processes (see Proposition~\ref{prop 2-1}).

\begin{hyp}\label{hypthe1-1}
  The centred Gaussian process $G$ with $G_0=0$ has a covariance function $R(s,t)=E (G_sG_t)$ that satisfies the following conditions:
  \begin{enumerate}
    \item[($H_1$)] $R(s,t)$ is an absolutely continuous function of $t\in [0,T]$ for any fixed $s \in[0,T]$;
    \item[($H_2$)]  for almost every $t \in[0,T]$, the first-order partial derivative
      $$\frac {\partial }{ \partial t}  R(s,t)$$
      is a normalized bounded variation (NBV) function of $s\in  [0,T]$.
  \end{enumerate}
\end{hyp} 

\begin{hyp}\label{hypthe2new-new}
  The centred Gaussian process $G$ with $G_0=0$ has a covariance function $R(s,t)=E (G_sG_t)$ that satisfies  Hypothesis~($H_1$) and the following condition:
  \begin{enumerate}
    \item[($H_3$)] for almost every $t \in[0,T]$,
      the difference
      \begin{align}
        \frac {\partial R(s,t)}{\partial t} - \frac {\partial R^{B}(s,t)}{\partial t} \label{cha-key}
      \end{align}
      is a normalized bounded variation function of $s\in  [0,T]$, where $R^B(s,t)$ is the covariance function of the fractional Brownian motion (fBm) $B^H=\set{B^H_t:\, t\in[0,T ]}$ with Hurst parameter $H\in (0, \,1)$.
  \end{enumerate}
\end{hyp}
Hypothesis ($H_2$) and Hypothesis ($H_3$) 
respectively imply that for almost every $t \in[0,T]$, \begin{equation}\label{papr dereiv}
      \frac{\partial R(\cdot,t)}{  \partial t } \text{\quad or \quad}  \frac{\partial R(\cdot,t)}{\partial t} -  \frac{\partial R^B(\cdot,t)}{\partial t}  
  \end{equation} determines a Lebesgue-Stieltjes signed measure on $([0,T], \mathcal{B}([0,T ]))$ and hence the measure can be used to define Lebesgue-Stieltjes integral 
  (see Section~\ref{fenbujfgsh subsect}), where $\mathcal{B}([0,T ])$ is the
Borel $\sigma$-algebra on $[0,T]$. The second integral in the iterated integrals \eqref{innp fg3-zhicheng0-0} and \eqref{innp fg3-zhicheng0} are such Lebesgue-Stieltjes integral of the normalized bounded variation function. 

Lemma~\ref{prop 2-1} is devoted to two variations
on the inner product formulas for  the {canonical Hilbert space associated to} $G$ under Hypotheses~\ref{hypthe1-1} and \ref{hypthe2new-new}.
We need several notations to state it. Denote $\mathcal{V}_{[0,T]}$ the set of functions on $[0,T]$ with bounded variation. Denote by $m$ the Lebesgue measure on $\R$. Denote by $D(g)$ the discontinuous points of a function $g$ on $[0,T]$ and by $NBV[0,T]$ the set of normalized bounded variation functions on $[0, T]$. 
For any function $g$ on $[0,T]$, denote
\begin{align*}
    A_1(g)&:=\set{t\in [0,T]: D(g)\cap D\big(\frac {\partial R(\cdot,t)}{\partial t}\big)\neq \emptyset,\,\frac {\partial R(\cdot,t)}{\partial t} \in NBV[0,T]},\\
    A_2(g)&:=\set{t\in [0,T]: D(g)\cap D\big(\frac{\partial R(\cdot,t)}{\partial t} -  \frac{\partial R^B(\cdot,t)}{\partial t} \big)\neq \emptyset}\\
    &\quad \bigcap \set{t\in [0,T]:\frac{\partial R(\cdot,t)}{\partial t} -  \frac{\partial R^B(\cdot,t)}{\partial t} \in NBV[0,T]}.
\end{align*} Recall that the sets of Lebesgue measure zero are called null sets.

\begin{lemma}\label{prop 2-1}
 Let $f,\, g\in \mathcal{V}_{[0,T]}$ and $R(s,t)=\E[G_sG_t],\,R^B(s,t)=\E[B^H_sB^H_t]$.
  \begin{enumerate}
    \item[(i)] If the covariance function $R(s,t)$ satisfies Hypothesis~\ref{hypthe1-1} and if 
    the set $A_1(g)$ is of Lebesgue measure zero (i.e., $A_1(g)$ is a null set),
    then
      \begin{align} \label{innp fg3-zhicheng0-0}
        \langle f,\,g \rangle_{\FH}=\int_0^T  f(t)\dif t \int_0^T g(s)  \dif \big(  \frac{\partial R(s,t)}{  \partial t }\big).
      \end{align}
    \item[(ii)] If the covariance function $R(s,t)$ satisfies Hypothesis~\ref{hypthe2new-new} and if  
    the set $A_2(g)$ is of Lebesgue measure zero (i.e., $A_2(g)$ is a null set),
    then
      \begin{align} \label{innp fg3-zhicheng0}
        \langle f,\,g \rangle_{\FH}-\langle f,\,g \rangle_{\FH_1}=\int_0^T   f(t)\dif t \int_0^T g(s)\dif \left( \frac{\partial R(s,t)}{\partial t} -  \frac{\partial R^B(s,t)}{\partial t} \right),
      \end{align}
      where $\mathfrak{H}_1$ is the {canonical Hilbert space
 associated to the fBm $B^H$.}
  \end{enumerate}
\end{lemma}
Proofs of Lemma~\ref{partial integral} and Lemma~\ref{prop 2-1} are given in Subsection~\ref{sub.sec.3.3}. 
\subsection{Proof of Proposition~\ref{prop 2-4 fenxilizi}}\label{gjzhenm sec.3.2}
{\bf Proof of Proposition~\ref{prop 2-4 fenxilizi}:} We need only to check the validity of the conditions of Lemma~\ref{prop 2-1}. We divide it into two steps. Step 1 is to check Hypothesis~\ref{hypthe1-1} or Hypothesis~\ref{hypthe2new-new}; and Step 2 is to check $m(A_1(g))=0$ or $m(A_2(g))=0$.

Step 1. (i) For Example~\ref{exmp 000001}, it is clear that when $s\ge 0, t\ge 0$,
\begin{align*}
    R(s,t)=H\abs{2H-1} \int_0^t \dif v\int_0^s (v+u)^{2H-2}\dif u,
\end{align*}
and when $t> 0$
\begin{align*}
    \frac{\partial R(s,t)}{  \partial t }=H\abs{2H-1}\int_0^s (t+u)^{2H-2}\dif u=\left\{
  \begin{array}{ll}
    H \big(t^{2H-1}- (t+s)^{2H-1}\big),          & \quad H\in (0,\,\frac12), \\
     H \big((t+s)^{2H-1}-t^{2H-1}\big), & \quad H\in (\frac12,\,1),
  \end{array}
  \right.
\end{align*}which implies that $R(s,t)$ is an absolutely continuous function on $t\in [0,T]$ for any fixed $s \in[0,T]$ and when $t>0$, the function $ \frac{\partial R(s,t)}{  \partial t }$ is an absolutely continuous function on $s\in [0,T]$. Hence, 
Examples~\ref{exmp 000001} satisfies Hypothesis~\ref{hypthe1-1}.

For Examples~\ref{exmp 000001-04}, it is clear that 
\begin{align*}
    R(s,t)&=(2H')^2 K (1-K) \int_0^t v^{2H'-1}\dif v\int_0^s (v^{2H'}+u^{2H'})^{K-2}u^{2H'-1}\dif u, \quad s,t\ge 0,\\
     \frac{\partial R(s,t)}{  \partial t }&=(2H')^2 K (1-K) t^{2H'-1} \int_0^s (t^{2H'}+u^{2H'})^{K-2}u^{2H'-1}\dif u, \quad \quad t>0,
\end{align*} which implies that  Examples~\ref{exmp 000001-04} satisfies Hypothesis~\ref{hypthe1-1} 
in the same vein.

 (ii) For Example~\ref{exmp 000001-06}, it is clear that
 \begin{align*}
     R(s,t)=\int_0^t \frac {\partial }{ \partial v}  R(s,v)\dif v,
 \end{align*}
and for $t>0$, the first-order partial derivative
  \begin{align*}
  \frac {\partial }{ \partial t}  R(s,t)&=\left\{
        \begin{array}{ll}
  H(t+s)^{2H-1}-H t^{2H-1}, & \quad 0\le s< t,\\
  H(t+s)^{2H-1}, & \quad t<s\le T
   \end{array}
  \right. \\
  &=\varphi(s)+\phi(s),\quad s\neq t,
  \end{align*}
where 
\begin{equation}\label{varphislizi}
  \varphi(s)
  =H(2H-1)\int_0^s (t+u)^{2H-2} \dif u,\quad s\in[0,T],
\end{equation} is an absolutely continuous function and 
\begin{equation}\label{phislizi}
  \phi(s)=\left\{
  \begin{array}{ll}
    0,          & \quad 0<s\le t, \\
    H t^{2H-1}, & \quad t<s\le T,
  \end{array}
  \right.
\end{equation} is a step function. Hence, $R(s,t)$ is an absolutely continuous function on $t\in [0,T]$ for any fixed $s \in[0,T]$, and for $t> 0$, $\frac {\partial }{ \partial t}  R(s,t)$ is a bounded variation function on $s\in [0,T]$.

(iii) For Example~\ref{exmp 000001-06-buch}, it is clear that
 \begin{align*}
     R(s,t)=\int_0^t \frac {\partial }{ \partial v}  R(s,v)\dif v,
 \end{align*}
and for $t>0$, the first-order partial derivative
  \begin{equation}\label{exmp 000001-06-piandao-new}
  \frac {\partial }{ \partial t}  R(s,t)=\left\{
        \begin{array}{ll}
  0 , & \quad 0<s\le t,\\
  2H\Gamma(1-2H) t^{2H-1}, & \quad t<s\le T
   \end{array}
  \right. 
  \end{equation} 
  is a pure jump function on $s\in [0,T]$. 

The fBm $B^H$ is involved in Examples~\ref{exmp0005}-~\ref{exmp lizi001}. Recall that the covariance function of $B^H$ is given by
 $$R^B(s,t)=\frac{1}{2}[s^{2H}+t^{2H}-\abs{t-s}^{2H}].$$
 It is clear that  \begin{align*}
     R^B(s,t)=\int_0^t \frac {\partial }{ \partial v}  R^B(s,v)\dif v,
 \end{align*}
and for $t>0$, the first-order partial derivative is given by
  \begin{equation*}
  \frac {\partial }{ \partial t}  R^B(s,t)=H\big[t^{2H-1}-\abs{t-s}^{2H-1}\sgn(t-s)\big],
  \end{equation*} which is of bounded variation when $H\in (\frac12, 1)$ and is not of  bounded variation when $H\in (0,\,\frac12)$.\\
(iv) For Example~\ref{exmp0005}, it is clear that for $s,t\ge 0$,
\begin{align*}
    R(s,t)=\int_0^t \frac{\partial R(s,v)}{\partial v} \dif v,  
\end{align*}
and for $t>0$, the first-order partial derivative is given by
  \begin{equation*}
  \frac {\partial }{ \partial t}  R(s,t)=H\big[2(s^{2H'}+t^{2H'})^{K-1}t^{2H'-1}-(t+s)^{2H-1}-\abs{t-s}^{2H-1}\sgn(t-s)\big] ,
 \end{equation*} which is of bounded variation when $H=H'K\in (\frac12, 1)$ and is not of bounded variation when $H\in (0,\,\frac12)$.

The first-order partial derivative of the difference function \eqref{cha-key} satisfies when $t>0$,
\begin{align}
  \quad&\frac{\partial R(s,t)}{\partial t  }- \frac{\partial  R^B(s,t)}{\partial t  }\nonumber\\
  &=H \big[ 2(s^{2H'}+t^{2H'})^{K-1}t^{2H'-1}-(t+s)^{2H-1}-t^{2H-1}\big]\notag\\
  &=H\int_0^s  \big[4H'(K-1)(u^{2H'}+t^{2H'})^{K-2}(tu)^{2H'-1} -(2H -1)(u+t)^{2H -2}\big]\dif u. \label{guanjiandengshiexmp1-5}
\end{align}
Therefore, $R(s,t)$ is an absolutely continuous function of $t\in [0,T]$ for any fixed $s \in[0,T]$ and when $t>0$, the difference function $ \frac{\partial R(s,t)}{\partial t  }- \frac{\partial  R^B(s,t)}{\partial t  }$ is an absolutely continuous function of $s\in [0,T]$ when $H\in (0,1)$. Hence, 
Examples~\ref{exmp0005} satisfies Hypothesis~\ref{hypthe2new-new}.

For Example~\ref{exmp7-1zili},  it is clear that for $s,t\ge 0$,
\begin{align*}
    R(s,t)=\int_0^t \frac{\partial R(s,v)}{\partial v} \dif v,  
\end{align*}
and for $t>0$, the first-order partial derivative is given by
\begin{align*}
 \frac{\partial R(s,t)}{\partial t}
 =H\left[\frac{(a+b)^2}{a^2+b^2}t^{2H-1}-\frac{2ab}{a^2+b^2}(t+s)^{2H-1}-|t-s|^{2H-1}\sgn(t-s)\right].
\end{align*} 
Thus, the first-order partial derivative of the difference function \eqref{cha-key} satisfies 
\begin{align}
  \frac{\partial R(s,t)}{\partial t  }- \frac{\partial  R^B(s,t)}{\partial t  } = -\frac{2H(2H-1)ab}{a^2+b^2} \int_0^s (t+u)^{2H-2} \dif u, \label{guanjiandengshiexmp1-6}
\end{align} which implies Examples~\ref{exmp7-1zili} satisfies Hypothesis~\ref{hypthe2new-new} 
in the same vein.

(v) For Example~\ref{exmp lizi001}, it clear that for $s,t\ge 0$,
\begin{align*}
     R(s,t)=\int_0^t \frac{\partial R(s,v)}{\partial v} \dif v,  
\end{align*}
and for $t>0$, the first-order partial derivative is given by
\begin{align*}
  \frac{\partial R(s,t)}{\partial t  }=\left\{
  \begin{array}{ll}
    H\left[t^{2H-1}-(t-s)^{2H-1}\right],           & \quad 0<s\le t, \\
    H (s-t)^{2H-1}, & \quad t<s\le T.
  \end{array}
  \right.
\end{align*}
Hence, the first-order partial derivative of the difference function \eqref{cha-key} satisfies 
\begin{equation}\label{guanjiandengshi 000-new1000}
  \frac{\partial R(s,t)}{\partial t  }- \frac{\partial  R^B(s,t)}{\partial t  }=\left\{
  \begin{array}{ll}
    0,           & \quad 0<s\le t, \\
    -H t^{2H-1}, & \quad t<s\le T
  \end{array}
  \right.
\end{equation}
is a step function. 

Step 2.
{ 
From the above calculations, we find that when $t>0$,  $\frac{\partial R(\cdot,t)}{\partial t}\in NBV[0,T]$ for Examples~\ref{exmp 000001}-\ref{exmp 000001-06-buch} and $\frac{\partial R(\cdot,t)}{\partial t} -  \frac{\partial R^B(\cdot,t)}{\partial t} 
      \in NBV[0,T]$ for Examples~\ref{exmp0005}-\ref{exmp lizi001}; and when $t>0$,  the set of discontinuous points of $\frac {\partial R(\cdot,t)}{\partial t}$ is
      \begin{align*}
         D\big(\frac {\partial R(\cdot,t)}{\partial t}\big)
        = \left\{
  \begin{array}{ll}
    \emptyset,       &  \text{for Examples~\ref{exmp 000001}-\ref{exmp 000001-04}}, \\
   \set{t}, & \text{for Examples~\ref{exmp 000001-06}-\ref{exmp 000001-06-buch}},
  \end{array}
  \right.
      \end{align*} and  the set of discontinuous points of $\frac {\partial R(\cdot,t)}{\partial t} -  \frac{\partial R^B(\cdot,t)}{\partial t}$ is
     \begin{align*}
         D\big(\frac{\partial R(\cdot,t)}{\partial t} -  \frac{\partial R^B(\cdot,t)}{\partial t}\big)
                = \left\{
  \begin{array}{ll}
    \emptyset,       &  \text{for Examples~\ref{exmp0005}-\ref{exmp7-1zili}}, \\
   \set{t}, & \text{for Examples~\ref{exmp lizi001}},
  \end{array}
  \right.
      \end{align*} 
which implies that
\begin{align*}
  &A_1(g)\subset{D(g)},\quad \text{for Examples~\ref{exmp 000001}-\ref{exmp 000001-06-buch}},\\
  &A_2(g)\subset{D(g)},\quad \text{for Examples~\ref{exmp0005}-\ref{exmp lizi001}}.
\end{align*} Since $g$ is of bounded variation, we have that $D(g)$ is a null set. 
Therefore,  $A_1(g)$ is a null set for Examples~\ref{exmp 000001}-\ref{exmp 000001-06-buch} and $ A_2(g)$ is null set for Examples~\ref{exmp0005}-\ref{exmp lizi001}.

Therefore, it follows from Proposition~\ref{prop 2-1} that the inner product formulas \eqref{diyineijigongshi 00}-\eqref{guanjiandengshi 000-new1} hold for Examples~\ref{exmp 000001}-\ref{exmp lizi001} respectively.  
  {\hfill\large{$\Box$}} 
\subsection{Proof of Lemma~\ref{partial integral} and Lemma~\ref{prop 2-1}}\label{sub.sec.3.3}
{In this subsection, we will use measure theory (especially, Lebesgue-Stieltjes signed measure) intensively for the proof of Lemma~3.1 and Lemma~3.3.}
Let us firstly recall some basic facts of measure theory which are necessary to the proof. $F$ is an absolutely continuous function on $t\in [0,T]$ if and only if $F$ is differentiable almost everywhere on $[0,T],\, F'\in L^1[0,T]$ and $F(x)-F(0)=\int_0^x F'(t)\dif t$; every bounded variation function is differentiable almost everywhere and is the difference of two bounded monotone functions. A basic jump function $J_{x_0}$ is a
function of the form
\begin{equation*}
J_{x_0}(x):=\left\{
      \begin{array}{lll}
      0,&\quad \text{when } x<x_0,\\
      a,&\quad \text{when } x=x_0,\\
1, &\quad \text{when } x>x_0   
 \end{array}
\right.
\end{equation*} 
for some $x_0\in \R$ and $0 \le a \le  1$. The function $J_{x_0}$ is also denoted by $\delta_{x_0}$ and are called the point mass at $x_0$.
A bounded variation function $F$ is said to be normalized if it is also right continuous on $\R$ and  $F(-\infty)=0$. See \cite{Foll99}.

For a normalized bounded variation function $g$ on $\R$, denote $\mu_g$ the signed Lebesgue-Stieltjes measure such that $\mu_g ((-\infty, x])=g(x)$. As usual, the Lebesgue-Stieltjes integral of a Borel measurable function $f$ with respect to the measure $\mu_{g}$
\begin{align}\label{benwenjifen}
    \int_{\R} f(x)\dif \mu_{g}(x)
\end{align}
is often abbreviated 
\begin{align}\label{benwenjifen 00}
    \int_{\R} f(x)\dif {g}(x).
\end{align}
When $g$ is absolutely continuous on~$\R$, the measure $\mu_g$ is absolutely continuous and the integral \eqref{benwenjifen 00} equals to 
$$\int_{\R} f(x) {g}'(x)\dif x.$$

For two absolutely continuous functions $f$ and $g$  on~$\R$, the well-known integration result (see \cite[p.108]{Foll99}) yields 
\begin{align}\label{foll-jingidan}
    -\int_0^T g(x) f'(x)\dif x=\int_0^T f(x) g'(x)\dif x+f(0)g(0)-f(T)g(T).
\end{align}
In general, if $F,G$ are of normalized bounded variation on $\R$ 
and if there are no points in $[0,T]$ where $F$ and $G$ are both discontinuous, then Exercise 3.34 (b) of \cite{Foll99} gives
\begin{align}\label{int by parts Folland 00}
 -\int_{[0,T]} G(x) \dif  F(x)=\int_{[0,T]} F(x)\dif G( x)-F(T)G(T)+F(0-)G(0-),
\end{align} where either $\mu_F$ or $\mu_G$ may has atom at endpoint $0$ or $T$. Especially, when $\mu_F$ and $\mu_G$ are absolutely continuous, the formula \eqref{int by parts Folland 00} degenerates to \eqref{foll-jingidan}.

To apply the formulas \eqref{foll-jingidan} and \eqref{int by parts Folland 00} to the case when $F,G$ are of normalized bounded variation on $[0,T]$, we need to firstly extend $F,G$ to $\R$ and then illustrate carefully the associated measures in \eqref{int by parts Folland 00},    since different extensions of $F,G$ lead to different mass on endpoints and thus different integral value, see (\ref{g0dingyi}) and (\ref{f01dingyi}) for example. In this paper, we choose  \eqref{g0dingyi} to extend the bounded variation function $G$ on $[0,T]$ to $\R$, which is cited from \cite{Jolis 2007}. Since $G(0-)=0$, the third term of the right side of \eqref{int by parts Folland 00} vanish. 

{ We are in a position to show Lemma~\ref{partial integral} and Lemma~\ref{prop 2-1}.}

{\bf Proof of Lemma~\ref{partial integral}:}
 We firstly describe the measure $\nu_{g}$ from a measure decomposition viewpoint (see \cite{cl2023}). The function $g^0(x)$ defined by \eqref{g0dingyi} is the difference of two bounded variation functions:
\begin{align}\label{basic decomp}
g^0(x)=g_1^0(x)-g_2^0(x),
\end{align}
where 
\begin{equation*}\label{nvb yantuo}
g_1^0(x)=\left\{
      \begin{array}{ll}
      0,&\quad \text{if } x<0,\\
 g(x), & \quad \text{if } x\in [0,T],\\
g(T), &\quad \text{if } x>T,    
 \end{array}
\right.
\end{equation*}
and 
\begin{equation*}
g_2^0(x)=\left\{
      \begin{array}{ll}
      0,&\quad \text{if } x<T,\\
g(T), &\quad \text{if } x>T.    
 \end{array}
\right.
\end{equation*}
Therefore, 
\begin{align}\label{ceducha}
\mu_{g^0} =\mu_{g^0_1}-\mu_{g^0_2}, 
\end{align} where 
\begin{equation}
\mu_{g^0_2}=g(T) \delta_T(\cdot),\label{g02 cedu}
\end{equation} 
and $\delta_T(\cdot)$ represents the Dirac measure at $T$. 


 Next, we extend $f$ to $\R$ as \eqref{f01dingyi}. Together with the assumption $f$ being normalized, we have $0$ is a continuous point of $\bar{f}$. Therefore, the identity \eqref{int by parts Folland 00} implies that if there are no points in $(0,T]$ where both $f$ and $g$ are discontinuous then
\begin{align*}
    -\int_{[0,T]} g(t) \dif \mu_{f}(t)&= -\int_{[0,T]} g_1^0(t) \dif \bar{f}(t)\notag\\
    &=\int_{[0,T]} \bar{f}(t) \dif g_1^0(t)-\bar{f}(T)g_1^0(T)+\bar{f}(0-)g_1^0(0-)\notag\\
    &=\int_{[0,T]} \bar{f}(t) \dif g_1^0(t)-\bar{f}(T)g_1^0(T)\\
    &=\int_{[0,T]} \bar{f}(t) \dif g_1^0(t)-\int_{[0,T]} \bar{f}(t) \dif g_2^0(t)\\
    &=\int_{[0,T]} \bar{f}(t) \dif  g^0( t) =\int_{[0,T]} f(t) \dif \nu_{ g}( t).
\end{align*}
 {\hfill\large{$\Box$}}

{\bf Proof of Lemma~\ref{prop 2-1}:}
We  only need to show (i) since (ii) is similar. The identity \eqref{jolis 00} together with Fubini's theorem implies that $\forall f,\, g\in \mathcal{V}_{[0,T]}$, \begin{align}  \label{inner product 001}
\innp{f,g}_{\FH}
=\int_{[0,T]} \Big( \int_{[0,T]}  R(t,s) \dif  \nu_f(t)  \Big)  \dif\nu_{g}( s).
\end{align}
 
 By Hypothesis~\ref{hypthe1-1} ($H_1$),   $R(s,t)$ is an absolutely continuous function on $t\in [0,T]$ for any fixed $s \in[0,T]$, then 
$\frac{\partial R(s,t)}{\partial t}$ exists almost everywhere on $t\in [0,T]$ for any fixed $s \in[0,T]$.
Applying the formula \eqref{01} to the function $R(\cdot, s)$ and $f(\cdot)$, we have  
  \begin{align} \label{zhjbuzhou 1}
    \langle f,\,g \rangle_{\FH}=-\int_{[0,T]} \Big(  \int_{[0,T]}f(t) \frac{\partial R(s,t)}{\partial t}  \dif t \Big)\dif \nu_{g}( s)=-\int_{[0,T]}  f(t)  \dif t \Big(\int_{[0,T]}\frac{\partial R(s,t)}{\partial t} \dif \nu_{g}( s)\Big).  
  \end{align}
Under Hypothesis~\ref{hypthe1-1} ($H_2$), the function
     $ \frac {\partial }{ \partial t}  R(s,t)$
      is a normalized bounded variation function of $s\in  [0,T]$ for  any fixed $t \in[0,T]$ almost everywhere. That is to say, 
\begin{equation*}
    m\Big(\set{t\in [0,T]:\frac {\partial }{ \partial t}  R(\cdot,t)
      \notin NBV[0,T]} \Big)=0,
\end{equation*} which, together with the condition $m(A_1(g))=0$, implies that the set 
\begin{equation*}
    B:=A_1(g)\bigcup \set{t\in [0,T]:\frac {\partial }{ \partial t}  R(\cdot,t)
      \notin NBV[0,T]}.
\end{equation*} is a null set, i.e.,  $m(B)=0$.
Hence 
\begin{equation}\label{zhjgdujieguo 00}
    \int_{[0,T]}  f(t)  \dif t \Big(\int_{[0,T]}\frac{\partial R(s,t)}{\partial t} \dif \nu_{g}( s)\Big)=\int_{[0,T]\setminus B}  f(t)  \dif t \Big(\int_{[0,T]}\frac{\partial R(s,t)}{\partial t} \dif \nu_{g}( s)\Big)
\end{equation}

For any fixed $t\in [0,T]\setminus B$, the function $\frac {\partial R(\cdot,t)}{\partial t}$ is of normalized bounded variation and $g(\cdot)$ has no common discontinuous points with the function $\frac {\partial R(\cdot,t)}{\partial t}$, and hence Lemma~\ref{partial integral} implies that 
\begin{align}\label{zhjjg 000}
   \int_{[0,T]}\frac{\partial R(s,t)}{\partial t} \dif \nu_{g}( s)=-\int_{0}^T g(s)\dif \big( \frac {\partial R(s,t)}{\partial t}\big). 
\end{align}      
 Substituting \eqref{zhjgdujieguo 00}-\eqref{zhjjg 000} into \eqref{zhjbuzhou 1} yields the integration by parts formula:
  \begin{align}
    \langle f,\,g \rangle_{\FH} & =\int_0^T  f(t) \dif t \int_0^T g(s) \dif \left(  \frac{\partial R(s,t)}{  \partial t }\right).
  \end{align}
 {\hfill\large{$\Box$}}

\section{Proof of {Theorem~\ref{asymptoticth}}}\label{proof mn results}
\begin{lemma} \label{upper bound F}
  Assume $\beta>0$ and $\theta>0$. Denote $$ {A}(s)=e^{-\theta s}\int_0^{s} e^{\theta r} r^{\beta -1}\dif r. $$ Then there exists a positive constant $C$ such that for any  $s\in [0,\infty)$,
  \begin{align*}
    {A}(s) & \le C \times\left(s^{\beta}\mathbbm{1}_{[0,1]}(s) + s^{\beta-1}\mathbbm{1}_{ (1,\,\infty)}(s)\right)\le  C \times \left(s^{\beta-1} \wedge s^{\beta} \right).
  \end{align*}
  Especially, when $\beta\in (0,1)$, there exists a positive constant $C$ such that for any  $s\in [0,\infty)$,
  \begin{align*}
    {A}(s) & \le C \times(1\wedge s^{\beta-1}).
  \end{align*}
\end{lemma}
{The proof of Lemma \ref{upper bound F} is trivial.  See \cite{chenzhou2021} for more details. 
\begin{lemma}\label{lem.4.2}
The covariance functions $R(s,t)$ of  Examples~\ref{exmp 000001}-\ref{exmp lizi001} satisfy four types of conditions. 
\begin{enumerate}
  \item[(T1)]
    Examples~\ref{exmp 000001}-\ref{exmp 000001-04} satisfy the following conditions:
    \begin{enumerate}  
    \item[($H_2'$)]  for almost every $t\in [0,T]$, the partial derivative
      $\frac {\partial R(s,t)}{\partial t} $
      is an absolutely continuous function of $s\in [0,T]$.
    \item[($H_4$)] 
        There exist constants $C_1,C_2\ge 0$ which depend only on $H',\,K$ such that the inequality
        \begin{equation}\label{phi2-new}
          \left| \frac{\partial}{\partial s}\left(\frac {\partial R(s,t)}{\partial t} \right)\right| \le  C_1  (t+s)^{2H-2}+C_2 (s^{2H'}+t^{2H'})^{K-2}(st)^{2H'-1}
        \end{equation}
        holds, where $H'>0,\, K\in (0,2)$ and $H:=H'K\in (0, 1)$.
    \end{enumerate}
  \item[(T2)]  
    Examples~\ref{exmp 000001-06}-\ref{exmp 000001-06-buch} satisfy  ($H_2$)  but {do} not satisfy  ($H_2'$).
  \item[(T3)] Examples~\ref{exmp0005}-\ref{exmp7-1zili} satisfy the following conditions:
    \begin{enumerate}
    \item[  ($H_{3}'$)] for almost every $t\in [0,T]$, the difference 
      $ \frac {\partial R(s,t)}{\partial t} - \frac {\partial R^{B}(s,t)}{\partial t}$ is an absolutely continuous function on $s\in [0,T]$.
      \item[($H_5$)]There exist constants $C_1,C_2\ge 0$ which depend only on $H',\,K$ such that the inequality
      \begin{equation}\label{phi2}
        \left| \frac{\partial}{\partial s}\left(\frac {\partial R(s,t)}{\partial t} - \frac {\partial R^{B}(s,t)}{\partial t}\right)\right| \le  C_1  (t+s)^{2H-2}+C_2 (s^{2H'}+t^{2H'})^{K-2}(st)^{2H'-1}
      \end{equation}
      holds, where $H'>0,\, K\in (0,2)$ and $H:=H'K\in (0, 1)$.
    \end{enumerate}
  \item[(T4)] Example~\ref{exmp lizi001} satisfies ($H_3$) but does not satisfy ($H_3'$).
\end{enumerate}
\end{lemma}}
\begin{remark}
The bi-fractional Brownian motion also satisfies the condition (T3). 
Recall that the covariance function of the bi-fractional Brownian motion $B^{H',K}$ with $H'>0,K \in (0, 2)$ and $H:=H'K\in (0, 1)$ is as follows:
      $$R(t,s)=\frac{1}{2}\left((s^{2H'}+t^{2H'})^K - |t-s|^{2H'K}\right). $$
See \cite{Bardina2011} for the case of $K\in (1, 2)$.
When the parameter of  $B^{H',K}$ is restricted to $H'\in (0,1),\,K\in (0,1]$, the statistical estimations for the Non-Vasicek model is studied in \cite{Sebaiy22021}. 
\end{remark}
{Proof of Lemma~\ref{lem.4.2} is included in the proof of Proposition 1.10, see Subsection~\ref{gjzhenm sec.3.2} for more details. 
\begin{lemma}\label{bds lem.2}
The right-hand sides of \eqref{phi2-new}-\eqref{phi2} have a upper bound: \begin{align}
 C_1  (t+s)^{2H-2}+C_2 (s^{2H'}+t^{2H'})^{K-2}(st)^{2H'-1}
 &\le   C\times (ts)^{H-1}.\label{uppppper}
\end{align}
\end{lemma} Lemma~\ref{bds lem.2} is trivial since $H'>0,\, K\in (0,2)$ and $H=H'K\in(0,1)$, and we ignore its proof.}
The following propositions are devoted to verify the conditions $(\mathcal{A}_1)$ and $(\mathcal{A}_3)$-$(\mathcal{A}_5)$ of Theorem~\ref{Eseiby theorem}, respectively. 
\begin{proposition}\label{proop3-1}
{Let the Gaussian process $G$ be taken from Examples~\ref{exmp 000001}-\ref{exmp lizi001}.
Then} there exists a positive constant $C$ that is independent of $T$, satisfying
  \begin{align}\label{gima2fang jie}
    \sigma^2(s,t):=\E\big[(G_s-G_t)^2\big]\le C\abs{s-t}^{2H},\quad 0\le s, t\le T,
  \end{align}where $\sigma^2(s,t)$ and $\sigma(s,t)$ are called the structure function and canonical metric for $G$, {respectively}.
\end{proposition}
\begin{proof}
  Suppose $0\le s<t\le T$. By taking $g(\cdot)=\mathbbm{1}_{[s,t]}(\cdot)$, we have
  \begin{align}\label{kaishidian}
     \sigma^2(s,t)=\E\big[(G_s-G_t)^2\big]=\norm{g}_{\FH}^2,\qquad \E\big[(B^H_s-B^H_t)^2\big]=\norm{g}_{\FH_1}^2. 
  \end{align}  
It is well-known that for the structure function of the fBm $B^H$ is
  \begin{align}\label{wellknownn}
    \E\left[(B^H_s-B^H_t)^2\right]=\norm{g}_{\FH_1}^2=\abs{s-t}^{2H}.
  \end{align}
  (i) The inner product \eqref{diyineijigongshi 00} of Examples~\ref{exmp 000001}-\ref{exmp 000001-04} implies that
  \begin{align}\label{gnorm2 diyilizi}
    \norm{g}_{\FH}^2 & =\int_{[s,t]^2}  \frac {\partial^2 R(u,v)}{\partial u\partial v} \dif u\dif v.
  \end{align}
  The inequalities \eqref{phi2-new} and \eqref{uppppper} imply that 
  \begin{align*}
      \abs{\frac {\partial^2 R(u,v)}{\partial u\partial v} }\le C(uv)^{H-1},
  \end{align*}where $C$ is a positive constant independent of $T$. Substituting it into \eqref{gnorm2 diyilizi}, we have 
\begin{align*}
    \sigma^2(s,t)\le C  \int_{[s,t]^2}  (uv)^{H-1}\dif u\dif v \le C(t-s)^{2H}.
\end{align*}
(ii) The inner product formula \eqref{guanjiandengshi 000-new1-000} of Example~\ref{exmp 000001-06} implies that
  \begin{align}\label{gnorm2 di2 lizi}
   \norm{g}^2_{\FH}= H(2H-1) \int_{[s,t]^2}  (u+v)^{2(H-1)}\dif u\dif v + H\int_{s} ^t u^{2H-1}\dif u.  
  \end{align}
Since $H\in (0,\frac12)$, the first term of the right hand side of \eqref{gnorm2 di2 lizi} is negative. Therefore, 
  \begin{align*}
  \norm{g}^2_{\FH} & \le H\int_{s} ^t u^{2H-1}\dif u= \frac12 t^{2H}\left(1-\left(\frac{s}{t}\right)^{2H}\right) \le \frac12(t-s)^{2H}.
  \end{align*}
  In the same vein, the inner product formula \eqref{guanjiandengshi 000-new1-000-buch} of Example~\ref{exmp 000001-06-buch} implies that
  \begin{align*}
   \norm{g}^2_{\FH} & = 2H\Gamma(1- 2H)\int_{s} ^t u^{2H-1}\dif u  \le \Gamma(1-2H) (t-s)^{2H}.
  \end{align*}Substituting the above two inequalities into \eqref{kaishidian}, we have that \eqref{gima2fang jie} holds for Example~\ref{exmp 000001-06} and \ref{exmp 000001-06-buch}.\\
  (iii) The inner product formula \eqref{dierneijigongshi 00} of Examples~\ref{exmp0005}-\ref{exmp7-1zili}  implies that
  \begin{align}\label{di34lizichfd}
    \norm{g}_{\FH}^2 -\norm{g}_{\FH_1}^2& ={\int_{[s,t]^2} \frac{\partial}{\partial u}\left( \frac{\partial R(u,v)}{\partial v} -  \frac{\partial R^B(u,v)}{\partial v} \right)\dif u\dif v}.
  \end{align}The inequalities \eqref{phi2} and \eqref{uppppper} imply that 
  \begin{align*}
      \abs{\frac {\partial^2 R(u,v)}{\partial u\partial v} -\frac {\partial^2 R^B(u,v)}{\partial u\partial v}}\le C(uv)^{H-1},
  \end{align*}where $C$ is a positive constant independent of $T$. Substituting it into \eqref{di34lizichfd}, we have that
\begin{align*}
    \abs{ \norm{g}_{\FH}^2 -\norm{g}_{\FH_1}^2}\le C  \int_{[s,t]^2}  (uv)^{H-1}\dif u\dif v \le C(t-s)^{2H}.
\end{align*}
  The triangle inequality together with \eqref{wellknownn} implies 
  \begin{align*}
     \norm{g}_{\FH}^2\le C(t-s)^{2H}.
  \end{align*}
  Hence, \eqref{gima2fang jie} holds for Examples~\ref{exmp0005}-\ref{exmp7-1zili}.\\
  (iv) The inner product formula \eqref{guanjiandengshi 000-new1} of Example~\ref{exmp lizi001} implies that
  \begin{align*}
  0\le \norm{g}_{\FH_1}^2  - \norm{g}_{\FH}^2& =H \int_{s}^t  u^{2H-1}\dif u \le \frac12\abs{s-t}^{2H}.
  \end{align*} In the same vein, \eqref{gima2fang jie} holds for Example~\ref{exmp lizi001}.
\end{proof}
\begin{corollary} 
{Let $\hat{\theta}_T$,\, $\hat{\mu}_T$ and $\hat{\alpha}_T$ be defined by \eqref{LSE1}, \eqref{LSE2} and \eqref{LSE3}, respectively.}
Then $\hat{\theta}_T$,\, $\hat{\mu}_T$ and $\hat{\alpha}_T$ are strong consistent. Moreover, the following integral is finite:
  \begin{align}\label{Sigmapingfang 2}
    \sigma^2_{\infty}:= 
    \E\left[\left(\int_0^{\infty} e^{-\theta s} G_s \dif s\right)^2\right]<\infty.
  \end{align}
\end{corollary}
\begin{proof}
Proposition~\ref{proop3-1} implies condition $(\mathcal{A}_1)$ of Theorem~\ref{Eseiby theorem} is valid. By Theorem~\ref{Eseiby theorem} (i), we have that $\hat{\theta}_T$,\, $\hat{\mu}_T$ and $\hat{\alpha}_T$ are strong consistent.

Taking $s=0$ in \eqref{gima2fang jie}, we have that there exists a positive constant $C$ independent of $T$ such that
  \begin{align*}
    \E\left[ G_t^2\right]\le Ct^{2H},\quad   t\ge 0.
  \end{align*}
  It follows from the Cauchy-Schwarz inequality and $\E[G_t]\equiv 0$ that the autocovariance function $\mathrm{cov}(G_t,G_s)$ of $G$ satisfies
  \begin{align*}
      \abs{\mathrm{cov}(G_t,G_s)}=\abs{\E[G_tG_s] }\le \Big(\E\left[ G_t^2\right]\E\left[ {G_s^2}\right]\Big)^{\frac12}\le  C(ts)^{H}.
  \end{align*}
Therefore, the Fubini's theorem implies that
  \begin{align*}
    \E\left[\left(\int_0^{\infty} e^{-\theta s} G_s \dif s\right)^2\right]&= \int_0^{\infty} \dif s \int_0^{\infty} \E[G_tG_s]\dif t \le C  \int_0^{\infty} \dif s \int_0^{\infty} \dif t\,e^{-\theta (s+t)} (ts)^{H}<\infty.
  \end{align*}
\end{proof}
Denote
\begin{align}\label{zeta T jifen}
  \zeta_t = \int_0^t e^{-\theta(t-s)} \dif G_s,\qquad    \eta_t = \int_0^t e^{-\theta(t-u)}\dif B^H_u,
\end{align}
where $\theta>0$ is a constant {and  $G$ is taken from Examples~\ref{exmp 000001}-\ref{exmp lizi001}}. Then $\zeta_t,\,\eta_t$ are respectively the ergodic OU process defined by the solutions to the stochastic differential equations 
\begin{align}
  \dif \zeta_t & =-\theta \zeta_t \dif t +\dif  G_t, \quad \zeta_0 = 0,\label{OU dingyi}       \\
  \dif \eta_t  & =-\theta \eta_t \dif t +\dif  B^H_t, \quad \eta_0 = 0.\label{OU dingyi duibi}
\end{align}
Propositions \ref{proop 3-5}-\ref{assum5} are concerning to the asymptotic growth of the covariance functions $\E(G_t\zeta_t),\, \E(G_s\zeta_t)$ and the variance function $\E(\zeta_t^2)$ as $t\to \infty$.

\begin{proposition}\label{proop 3-5}
  Let $\beta$ and $\zeta_t$ be defined by \eqref{beta biaoshishizi} and \eqref{zeta T jifen}, respectively. Then we have
  \begin{align}\label{assum-chy-52-001}
    \lim_{t \to \infty}  t^{2\beta}\E \left( \zeta_t ^2\right) & =\sigma_G^2, 
  \end{align}
  where
  \begin{equation}\label{changshu sg2}
    \sigma_G^2=\left\{
    \begin{array}{ll}
      \frac{H\abs{2H-1}2^{2H-2}}{\theta^2}, & \text{for Example~\ref{exmp 000001}},                  \\
      \frac{ 2^K K(1-K)(H')^2}{\theta^2} ,  & \text{for Example~\ref{exmp 000001-04}},               \\
      \frac{H}{2\theta} ,                   & \text{for Example~\ref{exmp 000001-06}},               \\
      \frac{H\Gamma(1-2H)}{\theta} ,        & \text{for Example~\ref{exmp 000001-06-buch}},          \\
      \theta^{-2H}H\Gamma(2H),              & \text{for Examples~\ref{exmp0005}-\ref{exmp lizi001}}.
    \end{array}
    \right.
  \end{equation}
\end{proposition}
\begin{proof}
  We assume that $\theta=1$ without loss of generality. Denote the function 
 \begin{equation}\label{func.h}
      h(\cdot)= e^{-(t-\cdot)} \mathbbm{1}_{[0,t]}(\cdot).
  \end{equation} 
{Then the random variable $\zeta_t$ defined by \eqref{zeta T jifen} can be rewritten as:}
  \begin{equation}\label{solution OU}
       \zeta_t=\int_0^t h(s) \dif G_s.
  \end{equation}
 By It\^{o}'s isometry, we have
  \begin{align}
      \E \left( \zeta_t ^2\right) =\norm{h}_{\FH}^2, \qquad \E \left( \eta_t ^2\right) =\norm{h}_{\FH_1}^2,
  \end{align}where $\eta_t=\int_0^t h(s) \dif B^H_s$ is the solution to \eqref{OU dingyi duibi}.\\
  (i) 
  The inner product \eqref{diyineijigongshi 00} of Examples~\ref{exmp 000001}-\ref{exmp 000001-04} implies that
\begin{equation}\label{hmo2}
    \norm{h}_{\FH}^2={\int_{[0,t]^2} e^{u+v-2t} \frac {\partial^2 R(u,v)}{\partial u\partial v} \dif u\dif v}. 
\end{equation}
Substituting \eqref{ziziexmp11111} into \eqref{hmo2} and by symmetry, we have for Example~\ref{exmp 000001},
\begin{equation*}
    \norm{h}_{\FH}^2=2H\abs{2H-1}e^{-2t} \int_0^t e^u\dif u\int_0^u e^v (u+v)^{2H-2}\dif v.
\end{equation*} 
Therefore L'H\^{o}pital's rule implies that for Example~\ref{exmp 000001},
  \begin{align}
    \lim_{t\to\infty} t^{2-2H}\E \left( \zeta_t ^2\right) \notag& =\lim_{t\to\infty}\frac{2H\abs{2H-1}}{e^{2t} t^{2H-2}}{\int_0^t e^u\dif u\int_0^u e^v (u+v)^{2H-2}\dif v}\notag \\
    & =\lim_{t\to\infty}H\abs{2H-1} \int_{0}^t e^{v-t} (1+\frac{v}{t})^{2H-2} \dif v\notag\\
    &=\lim_{t\to\infty}H\abs{2H-1} \int_{0}^t e^{-x} (2-\frac{x}{t})^{2H-2} \dif x,\quad (\text{ by } x=t-v)\notag\\
    &=H\abs{2H-1}2^{2H-2},\label{zizi111jixian}
  \end{align}where the last line is from Lebesgue's dominate theorem. \\
 Similarly, substituting \eqref{liziexmp2222} into \eqref{hmo2} and by symmetry, we have for Example~\ref{exmp 000001-04},
 \begin{equation*}
        \norm{h}_{\FH}^2=2K(1-K)(2H')^2 e^{-2t} \int_0^t e^u u^{{2H'-1}}\dif u\int_0^u e^v  (u^{2H'}+v^{2H'})^{K-2} v^{{2H'-1}}\dif v.
 \end{equation*}
Therefore L'H\^{o}pital's rule and the integration by parts imply that for Example~\ref{exmp 000001-04}, 
  \begin{align*}
     & \lim_{t\to\infty}{t^{2-2H}} \E \left( \zeta_t ^2\right)  \\
     & =\lim_{t\to\infty}\frac{2K(1-K)(2H')^2}{e^{2t} t^{2H-2}}\int_0^t e^u u^{{2H'-1}}\dif u\int_0^u e^v  (u^{2H'}+v^{2H'})^{K-2} v^{{2H'-1}}\dif v  \\
     & =\lim_{t\to\infty}\frac{K(1-K)(2H')^2}{e^t t^{2H'(K-1)-1}}  \int_{0}^t e^{v} (t^{2H'}+v^{2H'})^{K-2} v^{2H'-1}\dif v                                                      \\
     & =\lim_{t\to\infty}\frac{K(1-K)(2H')^2}{e^t t^{2H'(K-1)-1}} \left[t^{2H'(K-1)-1}2^{K-2}e^t  +(K-2)t^{2H'-1}\int_{0}^t e^{v}(t^{2H'}+v^{2H'})^{K-3} v^{2H'-1} \dif v\right] \\
     & = 2^K K(1-K) (H')^2,
  \end{align*}
  where in the last line, we {use Lemma~\ref{upper bound F} and the following estimate}:
  \begin{align*}
    \frac{1}{t^{2H'(K-2)}}\int_{0}^t e^{u-t}(t^{2H'}+u^{2H'})^{K-3} u^{2H'-1} \dif u  \le \frac{1}{t^{2H'K}}\int_{0}^t e^{u-t} u^{2H'K-1} \dif u\le  \frac{C }{t }.
  \end{align*}
  (ii) 
  The inner product formula \eqref{guanjiandengshi 000-new1-000} of Example~\ref{exmp 000001-06} and the symmetry imply that 
  \begin{equation}\label{zijiexmp 1.3}
       \norm{h}_{\FH}^2=2H(2H-1) e^{-2t}\int_{0}^t e^{u}\dif u\int_0^u e^v  (u+v)^{2H-2}\dif u\dif v + H\int_{0} ^t e^{2u-2t}u^{2H-1}\dif u.
  \end{equation} The limit \eqref{zizi111jixian} implies that 
  \begin{align*}
      \lim_{t\to\infty } t^{1-2H} e^{-2t}\int_{0}^t e^{u}\dif u\int_0^u e^v  (u+v)^{2H-2}\dif u\dif v=0.
  \end{align*} It is clear 
  \begin{align}\label{lcear jixian}
     \lim_{t\to\infty } t^{1-2H}e^{-2t} \int_{0} ^t e^{2u}u^{2H-1}\dif u= \frac12.
  \end{align} Substituting the above two limits into \eqref{zijiexmp 1.3}, we have that for Example~\ref{exmp 000001-06}, 
  \begin{align*}
   \lim_{t\to\infty }  t^{1-2H}\E \left( \zeta_t ^2\right) = \frac{H}{2}.
  \end{align*}
The inner product formula \eqref{guanjiandengshi 000-new1-000} of Example~\ref{exmp 000001-06-buch} implies that 
  \begin{equation*}
     \norm{h}_{\FH}^2=2H\Gamma(1-2H)  \int_{0} ^t e^{2u-2t}u^{2H-1}\dif u.  
  \end{equation*}
By \eqref{lcear jixian}, we have for Example~\ref{exmp 000001-06-buch},
\begin{align*}
   \lim_{t\to\infty }  t^{1-2H}\E \left( \zeta_t ^2\right) =  {H}\Gamma(1-2H).
  \end{align*}
  (iii) 
  The inner product formula \eqref{dierneijigongshi 00} of Examples~\ref{exmp0005}-\ref{exmp7-1zili} implies that
  \begin{equation}
     \norm{h}_{\FH}^2-\norm{h}_{\FH_1}^2=\int_{[0,t]^2} e^{u-t+v-t} \frac{\partial }{\partial u}\left( \frac{\partial R(u,v)}{\partial v} -  \frac{\partial R^B(u,v)}{\partial v} \right) \dif u   \dif v. 
  \end{equation}
Using the inequalities \eqref{phi2}-\eqref{uppppper}, we have as $t\to\infty$
  \begin{align*}
   \abs{ \norm{h}_{\FH}^2-\norm{h}_{\FH_1}^2}& ={\int_{[0,t]^2} e^{u-t+v-t} \abs{\frac{\partial }{\partial u}\left( \frac{\partial R(u,v)}{\partial v} -  \frac{\partial R^B(u,v)}{\partial v} \right) }\dif u   \dif v} \\
   & \le  C\left(\int_0^t e^{u-t} u^{H-1}\dif u\right)^2\to 0,
  \end{align*} where in the last line we use Lemma~\ref{upper bound F}.
Furthermore, it is well-known (see \cite{hnz 19}) that
  \begin{align} \label{hnz19-1}
  \lim_{t\to\infty}\norm{h}_{\FH_1}^2
  = H\Gamma(2H).
  \end{align}
By the triangle inequality, we have for Examples~\ref{exmp0005}-\ref{exmp7-1zili}, 
  \begin{align*} \lim_{t\to\infty}\E[  \zeta_t^2] =\lim_{t\to\infty}\norm{h}_{\FH}^2= H\Gamma(2H).
  \end{align*}
  (iv) 
  The inner product formula \eqref{guanjiandengshi 000-new1} of Example~\ref{exmp lizi001} implies that
\begin{equation*}
     \norm{h}_{\FH}^2-\norm{h}_{\FH_1}^2= -H{\int_{0}^t e^{2(u-t) } u^{2H-1} \dif u   }.
\end{equation*}Using Lemma~\ref{upper bound F}, we have 
\begin{align*}
   \lim_{t\to\infty} [ \norm{h}_{\FH}^2-\norm{h}_{\FH_1}^2]=0
\end{align*} since $H\in (0,\frac12)$.
Combining it with \eqref{hnz19-1}, we have 
for Example~\ref{exmp lizi001},  \begin{align*} \lim_{t\to\infty}\E[  \zeta_t^2] = H\Gamma(2H).
  \end{align*}
\end{proof}


\begin{proposition}\label{assum5-1}
{Let the Gaussian process $G$ be taken from Examples~\ref{exmp 000001}-\ref{exmp lizi001}.
Then} the covariance function $\E[G_sG_t]$ satisfies for any fixed $s>0$,
  \begin{align}
    \lim_{t \to \infty}\frac{1}{t^{H}} \E[G_sG_t] & =0. \label{assum51}
  \end{align}
\end{proposition}
\begin{proof} Since $s>0$ is fixed and $t\to\infty$, we can take $t>2s$. Denote
  \begin{align} \label{fof jihao}
    f_0(\cdot) =  \mathbf{1}_{[0,s]} (\cdot),\quad f(\cdot) =  \mathbf{1}_{[0,t]} (\cdot),\quad g(\cdot) =  \mathbf{1}_{[s,t]} (\cdot).
  \end{align} It\^{o}'s isometry implies that 
  \begin{align*}
        \E[G_sG_t]=\E[G_sG_s]+\E[G_s(G_t-G_s)] =\langle f_0,\, f_0\rangle_{\FH}+\langle g,\, f_0\rangle_{\FH}.
  \end{align*}
  Thus, we only need to show that the following limit holds:
  \begin{align}
    \lim_{t\to \infty}\frac{1}{t^H}\langle g,\, f_0\rangle_{\FH}=0.\label{prop 3.3 mubiao}
  \end{align}
  (i) The inner product \eqref{diyineijigongshi 00} of Examples~\ref{exmp 000001}-\ref{exmp 000001-04} implies that
\begin{align}\label{zhzhzhzhjjjj}
    \langle g,\, f_0\rangle_{\FH}=\int_{s}^t\dif u \int_0^s {  \frac {\partial^2 R(u,v)} {\partial u\partial v} }\dif v.
\end{align}
Substituting the inequality \eqref{phi2-new} 
into \eqref{zhzhzhzhjjjj} yields 
\begin{align}
    \abs{ \langle g,\, f_0\rangle_{\FH}}&\le  {C} \int_0^s\dif u \int_s^t   (u+v)^{2H-2}+  (u^{2H'}+v^{2H'})^{K-2}(uv)^{2H'-1} \dif v\notag\\
    &=C\Big[ t^{2H}\left(\left(1+\frac{s}{t}\right)^{2H}-1\right)-(2s)^{2H}+s^{2H}\notag                                                                                             \\
          & + t^{2H}\left(\left(1+\left(\frac{s}{t}\right)^{2H'}\right)^{K}-1\right)-2^K s^{2H}+s^{2H}\Big] 
\end{align}
 Since $H'>0,\, K\in (0,2)$ and $H:=H'K\in (0, 1)$ and $\frac{s}{t}\in (0,\frac12)$, Taylor's formula implies that
  \begin{align*}
\frac{1}{t^{H} }\abs{ \langle g,\, f_0\rangle_{\FH}}& \le {C\left\{t^{H-1}\left[2Hs+ O\left(\frac{s}{t}\right)\right]+t^{H}\left[K\left(\frac{s}{t}\right)^{2H'}+ O\left(\left(\frac{s}{t}\right)^{4H'}\right)\right]+   t^{-H}\right\}}.
  \end{align*} Taking $t\to\infty$ yields \eqref{prop 3.3 mubiao} for Examples~\ref{exmp 000001}-\ref{exmp 000001-04}. 
 \\
  (ii) Since the intersection of the supports of $f_0$ and $g$ is of Lebesgue measure zero, the inner product formula \eqref{guanjiandengshi 000-new1-000} of Example~\ref{exmp 000001-06} implies that
  \begin{align*}
    \langle g,\, f_0\rangle_{\FH}=H(2H-1)\int_0^s\dif u \int_{s}^t (u+v)^{2H-2}  \dif v,
  \end{align*} which implies \eqref{prop 3.3 mubiao} for Example~\ref{exmp 000001-06} in the same way as (i).  
  Similarly, the inner product formula \eqref{guanjiandengshi 000-new1-000-buch} of Example~\ref{exmp 000001-06-buch} implies that 
   \begin{align*}
    \langle g,\, f_0\rangle_{\FH}\equiv 0
  \end{align*} which implies \eqref{prop 3.3 mubiao} for Example~\ref{exmp 000001-06-buch}. \\
  (iii) 
  The inner product formula \eqref{dierneijigongshi 00} of Examples~\ref{exmp0005}-\ref{exmp7-1zili}  implies that
  \begin{align}\label{zhzhjjjj004}
     \langle g,\, f_0\rangle_{\FH}- \langle g,\, f_0\rangle_{\FH_1}= \int_0^s\dif u \int_s^t { \frac{\partial^2 R(u,v)}{\partial u\partial v}-\frac{\partial^2 R^B(u,v)}{\partial u\partial v}} \dif v.
  \end{align}  
  Substituting the inequality \eqref{phi2} into \eqref{zhzhjjjj004} yields 
  \begin{align*}
    \abs{\langle g,\, f_0\rangle_{\FH}- \langle g,\, f_0\rangle_{\FH_1}} &\le C\int_0^s\dif u \int_s^t   (u+v)^{2H-2}+  (u^{2H'}+v^{2H'})^{K-2}(uv)^{2H'-1} \dif v.
  \end{align*}
In the same vein as (i), we have \begin{align}\label{jiejji 10000000}
    \lim_{t\to\infty} \frac{1}{t^{H}}\abs{\langle g,\, f_0\rangle_{\FH}- \langle g,\, f_0\rangle_{\FH_1}} =0.
\end{align}
  For the fBm $(B^{H})$ and its covariance function $R^B(s,t)=\E[B^H_sB^H_t]$, It\^{o}'s isometry \eqref{G extension defn} implies that  as $t\to \infty$,
  \begin{align} \label{taylar 10000000}
    \frac{1}{t^H}    \langle g,\, f_0\rangle_{\FH_1}=
    \frac{1}{t^H}   \E\big[B^H(s)\big(B^H(t)-B^H(s)\big)\big]=\frac{1}{2t^H}\big[t^{2H}-(t-s)^{2H}-s^{2H}\big]\to 0.
  \end{align}
The triangle inequality together with \eqref{jiejji 10000000} and \eqref{taylar 10000000} implies \eqref{prop 3.3 mubiao} for Examples~\ref{exmp0005}-\ref{exmp7-1zili}.\\
  (iv)  Since the intersection of the supports of $f_0$ and $g$ is of Lebesgue measure zero, the inner product formula \eqref{innp fg3-zhicheng0-00} of Example~\ref{exmp lizi001} implies that
  \begin{align*}
    \frac{1}{t^H}\abs{\langle g,\, f_0\rangle_{\FH}}=\frac{1}{t^H}    \langle g,\, f_0\rangle_{\FH_1}\to 0,
  \end{align*}where the last limit is by \eqref{taylar 10000000}. Hence  \eqref{prop 3.3 mubiao} holds for Example~\ref{exmp lizi001}.
\end{proof}

\begin{proposition}\label{assum5}
{Let the Gaussian process $G$ be taken from Examples~\ref{exmp 000001}-\ref{exmp lizi001}
and} let $\zeta_t$ be difined by \eqref{zeta T jifen}. We have that  for any fixed $s\ge 0$,
  \begin{align}
    \lim_{t \to \infty} \frac{1}{t^{H}}\E \left( G_t   \zeta_t \right) =0, \qquad
    \lim_{t \to \infty}  \E \left( G_s   \zeta_t \right) & =0. \label{assum52-1}
  \end{align}
\end{proposition}
\begin{proof}
  We assume that $\theta=1$ without loss of generality. Let  the functions $f,\,f_0$ be given by \eqref{fof jihao} and $h$ by \eqref{func.h}.
By \eqref{solution OU}, It\^{o}'s isometry \eqref{G extension defn} implies that
\begin{align*}
    \E \left( G_t   \zeta_t \right) =\langle h,\, f\rangle_{\FH},\qquad \E \left( G_s   \zeta_t \right)=\langle h,\, f_0\rangle_{\FH}.
\end{align*} 
Thus, it suffices to show that for any fixed $s\ge 0$,
  \begin{align}\label{jixian limit 00}
    \lim_{t\to \infty}\frac{1}{t^H}\langle h,\, f\rangle_{\FH}=0,\qquad  \lim_{t\to \infty} \langle h,\, f_0\rangle_{\FH}=0.
  \end{align}  
  (i) The inner product \eqref{diyineijigongshi 00} of Examples~\ref{exmp 000001}-\ref{exmp 000001-04} implies that
\begin{equation}
\begin{split}\label{zhzhzhzhjjjj-3}
   \langle h,\, f\rangle_{\FH} & =  {\int_{[0,t]^2} e^{u-t} \frac {\partial^2 R(u,v)}{\partial u\partial v} \dif u\dif v},\\
   \langle h,\, f_0\rangle_{\FH}           & =  {\int_0^t e^{u-t} \dif u \int_0^s \frac {\partial^2 R(u,v)}{\partial u\partial v}  \dif v}. 
   \end{split}
\end{equation}
Using the inequalities \eqref{phi2-new}-\eqref{uppppper} and Lemma~\ref{upper bound F}, 
we have as $t\to \infty$,
  \begin{align*}
    \frac{1}{t^H} \abs{\langle h,\, f\rangle_{\FH} } & \le \frac{ C}{t^H}\int_{[0,t]^2}  e^{u-t}u^{H-1}v^{H-1} \dif u\dif v \to 0, \\
    \abs{\langle h,\, f_0\rangle_{\FH} }           & \le C \int_{0}^t  e^{u-t}u^{H-1}  \dif u\int_0^s v^{H-1}\dif v \to 0.
  \end{align*} 
  Hence, \eqref{jixian limit 00} holds for Examples~\ref{exmp 000001}-\ref{exmp 000001-04}.\\
  (ii) The inner product formula \eqref{guanjiandengshi 000-new1-000} of Example~\ref{exmp 000001-06} implies that 
 \begin{align*}
   \langle h,\, f\rangle_{\FH} & =H(2H-1) {\int_{[0,t]^2} e^{u-t} (u+v)^{2H-2} 
   \dif u\dif v}+ H\int_{0} ^t e^{u-t}u^{2H-1}\dif u,\\
   \langle h,\, f_0\rangle_{\FH}           & =  H(2H-1) \int_{0}^t e^{u-t}\dif u \int_0^s  (u+v)^{2(H-1)}\dif v + H\int_{0} ^s e^{u-t}u^{2H-1}\dif u, 
\end{align*}where the integrals are finite since $H\in (0,\frac12)$.
Using the inequalities \eqref{phi2-new}-\eqref{uppppper} and Lemma~\ref{upper bound F}, we have \eqref{jixian limit 00} holds for Example~\ref{exmp 000001-06}. 
The inner product formula \eqref{guanjiandengshi 000-new1-000} of Example~\ref{exmp 000001-06-buch} implies that 
 \begin{align*}
    \langle h,\, f\rangle_{\FH} & = 2H\Gamma(1-2H)\int_{0} ^t e^{u-t}u^{2H-1}\dif u,\\
   \langle h,\, f_0\rangle_{\FH}           & =  2H\Gamma(1-2H)\int_{0} ^s e^{u-t}u^{2H-1}\dif u. 
\end{align*}
 since $H\in (0,\frac12)$, we have \eqref{jixian limit 00} holds for Example~\ref{exmp 000001-06-buch}.\\
(iii) The inner product formula \eqref{dierneijigongshi 00} of Examples~\ref{exmp0005}-\ref{exmp7-1zili} implies that
 \begin{equation*}
     \begin{split}
         \langle h,\, f\rangle_{\FH}-\langle h,\, f\rangle_{\FH_1}&={\int_{[0,t]^2} e^{u-t} \Big(\frac {\partial^2 R(u,v)}{\partial u\partial v}-\frac {\partial^2 R^B(u,v)}{\partial u\partial v} \Big)\dif u\dif v},\\
        \langle h,\, f_0\rangle_{\FH}-\langle h,\, f_0\rangle_{\FH_1}&={\int_0^t e^{u-t} \dif u \int_0^s \Big(\frac {\partial^2 R(u,v)}{\partial u\partial v}-\frac {\partial^2 R^B(u,v)}{\partial u\partial v}  \Big)\dif v}. 
     \end{split}
 \end{equation*}
Using the inequalities \eqref{phi2}-\eqref{uppppper} and Lemma~\ref{upper bound F}, 
we have as $t\to \infty$,
  \begin{equation}\label{jj-3new}
     \begin{split}
    \frac{1}{t^H} \abs{\langle h,\, f\rangle_{\FH}-\langle h,\, f\rangle_{\FH_1}} & \le \frac{ C}{t^H}\int_{[0,t]^2}  e^{u-t}u^{H-1}v^{H-1} \dif u\dif v \to 0, \\
    \abs{\langle h,\, f_0\rangle_{\FH}-\langle h,\, f_0\rangle_{\FH_1}}           & \le C \int_{0}^t  e^{u-t}u^{H-1}  \dif u\int_0^s v^{H-1}\dif v \to 0.
     \end{split}
 \end{equation}
 On the other hand, the inner product formula of the fBm $(B^H)$ (see \cite{Alazemi2024}) implies that
  \begin{align}
     \langle h,\, f\rangle_{\FH_1} & =   \int_{[0,t]^2}  e^{u-t}  \abs{u-v}^{2H-1}\sgn{(u-v)} \big(\delta_0(v)-\delta_t(v)\big)\dif v\dif u \notag\\& =    \int_0^t e^{  u-t}\Big(u^{2H-1}+(t-u)^{2H-1}\Big)\dif u,\notag\\
    \langle h,\, f_0\rangle_{\FH_1}              & = {H} \int_{[0,t]^2}  e^{u-t}   \abs{u-v}^{2H-1}\sgn{(u-v)}  \big(\delta_0(v)-\delta_s(v)\big)\dif v\dif u     \notag\\ & =H \left(\int_0^s+\int_s^t \right) e^{u-t}  \Big(u^{2H-1}-\abs{u-s}^{2H-1}\sgn{(u-s)}\Big) \dif u. \label{di 2222 bufen}
  \end{align}
Hence, Lemma~\ref{upper bound F} implies that  for any fixed $s>0$,
  \begin{align}\label{di333222}
       \lim_{t\to \infty}\frac{1}{t^H}\langle h,\, f\rangle_{\FH_1}=0.
  \end{align}It is clear that for any fixed $s>0$,
  \begin{align}
     \lim_{t\to\infty}  \int_0^s  e^{u-t}  \Big(u^{2H-1}-\abs{u-s}^{2H-1}\sgn{(u-s)}\Big) \dif u=0.\label{dji 1 bufen jix}
  \end{align} Since $t\to\infty $ and $s>0$ is fixed, we can take $t\ge 2s$, i.e., $\frac{s}{t}\in [0,\frac12]$. The elementary inequality 
  \begin{equation*}
      \abs{(1-x)^{2H-1}-1}\le C_H x,\qquad x\in [0,\frac12]
  \end{equation*}where $H\in (0,1)$ and $C_H$ is a positive constant, implies that for any fixed $s>0$,
  \begin{align*}
      \lim_{t\to\infty}[(t-s)^{2H-1}-t^{2H-1}]=0.
  \end{align*}By L'H\^{o}pital's rule, we have for any fixed $s>0$,
   \begin{align}
     &\lim_{t\to\infty}  \int_s^t  e^{u-t}  \Big(u^{2H-1}-\abs{u-s}^{2H-1}\sgn{(u-s)}\Big) \dif u\notag\\
     &=\lim_{t\to\infty} \frac{1}{e^t} \int_s^t  e^{u}  \Big(u^{2H-1}-(u-s)^{2H-1}\Big) \dif u\notag\\
     &=\lim_{t\to\infty}[(t-s)^{2H-1}-t^{2H-1}]=0. \label{djierbufen jix}
  \end{align}Combining \eqref{di 2222 bufen},\eqref{dji 1 bufen jix} and \eqref{djierbufen jix}, we have  for any fixed $s>0$,
  \begin{align}\label{di333}
       \lim_{t\to \infty} \langle h,\, f_0\rangle_{\FH_1}=0.
  \end{align} Combining \eqref{jj-3new}, \eqref{di333222} and \eqref{di333},
and using the triangle inequality, we have \eqref{jixian limit 00} holds for Examples~\ref{exmp0005}-\ref{exmp7-1zili}. \\
  (iv) The inner product formula \eqref{guanjiandengshi 000-new1} of Example~\ref{exmp lizi001} implies that
  \begin{equation*}
      \begin{split}
        {\langle h,\, f\rangle_{\FH}-\langle h,\, f\rangle_{\FH_1}} & = -\int_{0}^t  e^{u-t}u^{2H-1} \dif u , \\
    {\langle h,\, f_0\rangle_{\FH}-\langle h,\, f_0\rangle_{\FH_1}}           & =-H \int_{0}^s  e^{u-t}u^{2H-1} \dif u. 
      \end{split}
  \end{equation*}
Since $H\in(0,\frac12)$, Lemma~\ref{upper bound F} implies that for any fixed $s>0$,
\begin{align*}
    \lim_{t\to\infty} \frac{1}{t^H} \abs{\langle h,\, f\rangle_{\FH}-\langle h,\, f\rangle_{\FH_1}} =0,\qquad \lim_{t\to\infty}\abs{\langle h,\, f_0\rangle_{\FH}-\langle h,\, f_0\rangle_{\FH_1}} =0,
\end{align*} which, combining with \eqref{di333222} and \eqref{di333}, implies
 that \eqref{jixian limit 00} holds for Example~\ref{exmp lizi001}.
\end{proof}
{\bf Proof of Theorem~\ref{asymptoticth}: }
Propositions~\ref{proop3-1} implies that the condition $(\mathcal{A}_1)$ of Theorem~\ref{Eseiby theorem} is satisfied. It is clear that all the Gaussian processes given in Examples~\ref{exmp 000001}-\ref{exmp lizi001} are self-similar which implies the following limit \begin{equation}\label{assum51-new}
  \lim_{t \to \infty}\frac{1}{t^{2H}} R(t,t) = \lambda_G^2=\left\{
  \begin{array}{ll}
    \abs{1-2^{2H-1}},                  & \quad \text{for Examples~\ref{exmp 000001} and \ref{exmp 000001-06}}, \\
    {2-2^{K}},                         & \quad \text{for Example~\ref{exmp 000001-04}} ,                       \\
    \Gamma(1-2H),                      & \quad \text{for Example~\ref{exmp 000001-06-buch}},                   \\
    2^{K}-2^{2H'K-1} ,                 & \quad \text{for Example~\ref{exmp0005}},                              \\
    \frac{(a+b)^2- 2^{2H}ab}{a^2+b^2}, & \quad\text{for Example~\ref{exmp7-1zili}},                            \\
    \frac12 ,                          & \quad \text{for Example~\ref{exmp lizi001}},
  \end{array}
  \right.
\end{equation}holds,  which means the condition $(\mathcal{A}_2)$ of Theorem~\ref{Eseiby theorem} is valid. 

Proposition~\ref{proop 3-5} implies that  the condition $(\mathcal{A}_3)$ of Theorem~\ref{Eseiby theorem} is satisfied. The second limit of \eqref{assum52-1} implies that the condition $(\mathcal{A}_4)$ of Theorem~\ref{Eseiby theorem} is satisfied. The first limit of \eqref{assum52-1} and Proposition~\ref{assum5-1} imply that the condition $(\mathcal{A}_5)$ of Theorem~\ref{Eseiby theorem} is satisfied. 

In a word, all the conditions $(\mathcal{A}_1)$-$(\mathcal{A}_5)$ of Theorem~\ref{Eseiby theorem} are satisfied for Examples~\ref{exmp 000001}-\ref{exmp lizi001}. Therefore, Theorem~\ref{Eseiby theorem} implies the desired Theorem~\ref{asymptoticth}.
  
{ 
Remark~\ref{asymptoticth OU} is a special case of Theorem~\ref{asymptoticth} essentially. We can also give it an independent proof by Theorem~\ref{Esseiby 2}.
First, the assumption $(\mathcal{A}_1)$ implies that the assumptions $(\mathcal{H}_1)$-$(\mathcal{H}_2)$ hold. Next, the conditions $(\mathcal{A}_s)$-$(\mathcal{A}_4)$ of Theorem~\ref{Eseiby theorem} have been checked in Proof of Theorem~\ref{asymptoticth}. 
In a word, all the conditions $(\mathcal{H}_1)$-$(\mathcal{H}_2)$ and $(\mathcal{A}_3)$-$(\mathcal{A}_4)$ of Theorem~\ref{Esseiby 2} are satisfied for Examples~\ref{exmp 000001}-\ref{exmp lizi001}. Therefore, Theorem~\ref{Esseiby 2} implies the desired \eqref{ou jiexianfubu}.}
  {\hfill\large{$\Box$}}


\begin{thebibliography}{99}

    \bibitem{Alsenafi2021}Alsenafi A, Al-Foraih M, Es-Sebaiy K.   Least Squares Estimation for Non-ergodic Weighted Fractional Ornstein-Uhlenbeck Process of General Parameters. AIMS Math, 2021, {\bf 6}: 12780--12794

   \bibitem{Alazemi2024} Alazemi F, Alsenafi A, Chen Y,  Zhou H. Parameter Estimation for the Complex Fractional Ornstein-Uhlenbeck Processes with Hurst parameter $H\in (0,\,\frac12)$. Chaos, Solitons \& Fractals, 2024, \textbf{188}:1155562024

    \bibitem{Bardina2009} Bardina X, Bascompte D.  {A Decomposition and Weak Approximation of the Sub-fractional Brownian Motion}.  Departament de Matematiques, UAB, 2009


    \bibitem{Bardina2011} Bardina X, Es-Sebaiy K. An Extension of Bifractional Brownian Motion. {Comm Stoch Anal,}  2011, \textbf{5}: 333--340

    \bibitem{BGT2007} Bojdecki T, Gorostiza L, Talarczyk A. Some Extensions of Fractional Brownian Motion and Sub-fractional Brownian Motion Related to Particle Systems. {Electron Commun Probab,}  2007, \textbf{12}: 161--172

\bibitem{Dwithdingzhen}
	Chen Y, Ding Z, Li Y.  2024. Berry-Ess\'een bounds and almost sure CLT for the quadratic variation of a general Gaussian process. \emph{ Comm. Statist. Theory Methods.} 51(13): 3920-3939 

    \bibitem{cl2023} Chen Y, Li Y. The Properties of Fractional Gaussian Process and Their Applications. {arXiv: 2309.10415}
    
\bibitem{chenzhou2021}
Chen Y, Zhou H. 2021. Parameter Estimation for an Ornstein-Uhlenbeck Processes driven by a general Gaussian Noise.  \emph{Acta Mathematica Scientia}, 41B(2): 573–-595

    \bibitem{DW2016} Durieu O, Wang Y. From infinite urn schemes to decompositions of self-similar Gaussian process. {Electron Commun Probab,}  2016, \textbf{21}: 1--23

    \bibitem{Sebaiy2016} El Machkouri M, Es-Sebaiy K, Ouknine Y. Least Squares Estimator for Nonergodic Ornstein-Uhlenbeck Processes Driven by Gaussian Processes. {J Korean Statist Soc,}  2016, \textbf{45}: 329--341

    \bibitem{NoutyJourn2013} El-Nouty C, Journ\'e J. The Sub-bifractional Brownian Motion. {Studia Sci Math Hungar,}  2013, \textbf{50}: 67--121

    \bibitem{Sebaiy22021} Es-Sebaiy K, Es.Sebaiy M. Estimating Drift Parameters in A Non-ergodic Gaussian Vasicek-type Model. {Stat Methods App,}  2021, \textbf{30}: 409--436

    \bibitem{Foll99} Folland G.  {Real analysis}.  Vol. 40 of Modern techniques and their applications, 1999

\bibitem{hnz 19}Hu Y, Nualart D, Zhou H. 2019. Parameter estimation for fractional Ornstein-Uhlenbeck processes of general Hurst parameter, \emph{Stat. Inference Stoch. Process} 22: 111--142 

  \bibitem{huang 2012} Huang J, Huang M.  How much of the corporate-treasury yield spread is due to
credit risk? The Review of Asset Pricing Studies. 2012. \textbf{2} (2): 153--202

\bibitem{jiang 2014}Jiang H, Pan Y, Xiao W, Yang Q, Yu J. Asymptotic theory for explosive
fractional Ornstein-Uhlenbeck processes. Electronic Journal of Statistics, 2024, \textbf{18}(2):3931-3974

    \bibitem{Jolis 2007} 
Jolis M.  2007. On the Wiener integral with respect to the fractional Brownian motion on an interval.  { J. Math. Anal. Appl. } \textbf{330}(2): 1115--1127

    \bibitem{KX2023} Kuang N, Xie H. Least Squares Type Estimators for the Drift Parameters in the Sub-bifractional Vasicek Processes. {Infin Dimens Anal Quantum Probab Relat Top,}  2023, \textbf{26}: 2350004

    \bibitem{Lei2009} Lei P, Nualart D. A Decomposition of the Bi-fractional Brownian Motion and Some Applications. {Statist Probab Lett,}  2009, \textbf{79}: 619--624

    \bibitem{ma2013} Ma C. The Schoenberg-L\'{e}vy Kernel and Relationships among Fractional Brownian Motion, Bifractional Brownian Motion, and Others. {Theory Probab Appl,}  2009, \textbf{57}: 619--632

\bibitem{NT 19}Nourdin I, Tran T D. Statistical inference for Vasicek-type model driven by Hermite processes. Stochastic Processes and their Applications, 2019, \textbf{129}(10), 3774-3791

    \bibitem{Sgh2013} Sghir A. The Generalized Sub-fractional Brownian Motion. {Commun Stoch Anal,}  2013, \textbf{7}: 373--382

    \bibitem{Sgh2014} Sghir A. A Self-similar Gaussian Process. {Random Oper Stoch Equ,}  2014, \textbf{22}: 85--92

 \bibitem{tao2011}
Tao, T. 2011.  {An introduction to measure theory. Vol. 126 of Graduate Studies in Mathematics}. Providence: American Mathematical Society

\bibitem{Talarczyk2020} Talarczyk A. Bifractional Brownian Motion for $H>1$ and $2HK \le 1$. {Statist Probab Lett,}  2020, \textbf{157}: 108628

\bibitem{wang2023}Wang X. Xiao W. Yu J.  Modeling and forecasting realized volatility with
the fractional Ornstein-Uhlenbeck process. Journal of Econometrics, 2023, \textbf{232}(2), 389-415
\bibitem{wu2020} Wu S,  Dong Y, Lv W,  Wang  G.  Optimal asset allocation for participating contracts with mortality risk under minimum guarantee. Communications in Statistics - Theory and Methods 2020. \textbf{49} (14): 3481-–3497

\bibitem{xiaoyu2019} Xiao W,  Yu J. Asymptotic theory for estimating drift parameters in the fractional
vasicek model. Econometric Theory. 2019, \textbf{35}(1):198–-231 

\bibitem{xiaoyu2019a} Xiao W,  Yu J. Asymptotic theory for rough fractional Vasicek models. Economics
Letters  2019, \textbf{177}: 26–-29

\bibitem{wei2023} Wei C. Least Squares Estimation for A Class of Uncertain Vasicek Model and Its Application to Interest Rates. {Stat,}  2023, 1--19

\bibitem{Yuqian2020} Yu Q. Statistical Inference for Vasicek-type Model Driven by Self-similar Gaussian Processes. {Comm Statist Theory Methods,}  2020, \textbf{49}: 471--484

\bibitem{Zili17} Zili M. Generalized Fractional Brownian Motion. {Mod Stoch Theory Appl,}  2017, \textbf{4}: 15--24



\end{thebibliography}
\end{document}